%% file: main.tex
\documentclass{siamart251216}
\input{settings}

\title{Structure-Informed Bounds on the Kronecker Rank of Block-Structured Matrices}
\author{
  Allison Fuller\thanks{School of Mathematical and Statistical Sciences, Arizona State University,
  Tempe, AZ, United States (\email{tjfuller@asu.edu, mespanol@asu.edu}).}
  \and
  Malena I.\ Espa\~{n}ol\footnotemark[1]
  \and
  Misha E.\ Kilmer\thanks{Department of Mathematics, Tufts University, Medford, MA, United States
  (\email{misha.kilmer@tufts.edu})}
}
\date{}

\begin{document}

\maketitle

\begin{abstract}
We derive theoretical bounds on the Kronecker rank of block-structured matrices that possess
both inner and outer structure.  Building on the matrix-to-tensor and tensor-to-matrix framework
of Kilmer and Saibaba (\emph{SIAM J.\ Matrix Anal.\ Appl.}, 2022), we show that the Kronecker
rank of a matrix $\bA$ equals the dimension of the span of its distinct blocks (the
\emph{inner blockspan}), and equals as well the dimension of the corresponding span for a
permuted matrix $\bB$ whose inner and outer structures are interchanged.  We give two proofs of
this equality: one via a direct dimension argument, and one via an explicit isomorphism between
the outer blockspan of $\bA$ and the column space of the second-mode unfolding of the associated
tensor.  These results yield nested containments that translate structural information, such as Toeplitz,
Hankel, banded, or sparse patterns, into computable upper bounds on the Kronecker rank.  We also
establish a duality between element-level sparsity in $\bA$ and block-level sparsity in $\bB$.
Numerical experiments confirm the theory across several classes of structured matrices and, for
two sparse matrices drawn from the SuiteSparse collection, provide a structural explanation for
observed singular value decay.
\end{abstract}

\section{Introduction}
Large, structured matrices arise in many applications, including linear inverse problems and the discretization of ordinary and partial differential equations. Their size often makes storage and computation prohibitively expensive. One effective way to mitigate these costs is to express such matrices as a sum of Kronecker products, thereby exploiting Kronecker structure to reduce storage requirements and accelerate matrix-vector products. 

The classical approach to writing a matrix as a sum of Kronecker products originated with Van Loan and Pitsianis~\cite{loan1992approximation}, and was later expanded by Pitsianis~\cite{pitsianis1997kronecker}. They showed that any matrix $\bA\in\RR{(\ell m)}{(qn)}$ admits the exact representation
\begin{equation*}
    \bA = \sum_{j=1}^r \bC_j\otimes \bD_j,
\end{equation*}
where $\bC_j\in\RR{\ell}{q}$ and $\bD_j\in\RR{m}{n}$ are obtained from singular vectors of a reshuffled version of $\bA$. The integer $r$ is the \emph{Kronecker rank} of $\bA$. Despite its strong theoretical guarantees, this approach is often impractical at large scale, as it requires computing the singular value decomposition of a matrix comparable in size to $\bA$.

Motivated by these limitations, we focus on a structured subclass of matrices for which more scalable Kronecker-based representations may be obtained, namely, large \emph{block-structured matrices}. We use the term \emph{block matrix} to denote a matrix partitioned into submatrices, called \emph{blocks}, and assume throughout that all blocks are of equal size. For example, a matrix $\bA \in \RR{(\ell m)}{(qn)}$ may be viewed as an $\ell \times q$ block matrix with $\ell q$ blocks arranged in an $\ell$-by-$q$ grid, each of size $m\times n$. We adopt a broad notion of \emph{structure}, encompassing both pattern-based structure (e.g., Toeplitz or Hankel matrices) and sparsity-based structure (e.g., diagonal or banded matrices). A \emph{block-structured matrix} is therefore a block matrix whose blocks exhibit such properties.

Kilmer and Saibaba~\cite{kilmer2022matrixtensordecomp} introduced a tensor-based framework tailored to such block-structured matrices. Their method assembles the distinct blocks of $\bA$ into a tensor and approximates it using a Tucker decomposition; when mapped back to the matrix domain, this yields a sum of Kronecker products whose length is determined by the second-mode rank of the tensor. Several examples in that work show that highly structured block matrices can be approximated to within machine precision using only a small number of Kronecker terms. This empirical observation raises a natural question: can one determine \emph{a priori} bounds on the Kronecker rank from the structural properties of the matrix alone?

In this paper, we answer this question by deriving exact characterizations and computable upper bounds on the Kronecker rank of matrices with both outer and inner block structures. Our approach combines the invertible matrix-to-tensor mapping of~\cite{kilmer2022matrixtensordecomp} with subspace dimension arguments defined by the block layout, the internal structure of the blocks, and a stride permutation of $\bA$. The main contributions are as follows.
\begin{itemize}
    \item We show that the Kronecker rank of $\bA$ equals the dimension of the linear span of its distinct blocks (Theorem~\ref{thm:r2=dimC}).
    \item We prove that this dimension is invariant under the permutation that interchanges inner and outer structure, via both a direct dimension argument (Theorem~\ref{thm:dimA=dimB}) and an explicit isomorphism with the column space of the second-mode unfolding (Lemma~\ref{lemma:B_iso_col_A2}).
    \item We establish nested containments linking structural subspaces to the Kronecker rank, yielding computable upper bounds from Toeplitz, Hankel, banded, and related structure (Corollary~\ref{cor:nested}).
    \item We derive a sparsity-based bound and a duality between element-level sparsity in $\bA$ and block-level sparsity in the permuted matrix $\bB$ (Theorems~\ref{thm:sparse} and~\ref{thm:block_sparse}).
    \item We show, through numerical experiments, that these structural bounds give tight upper bounds that explain the singular value decay observed in~\cite{kilmer2022matrixtensordecomp} for specific matrices from the SuiteSparse collection~\cite{davis2011suite}.
\end{itemize}

Prior work on Kronecker rank has focused primarily on matrices whose entries are given by smooth functions~\cite{hackbusch2004kronecker, Tyrtyshnikov2004KPapprox}. In application-driven settings such as image deblurring, related work has examined representations of structured operators as sums of structured Kronecker products~\cite{kamm1998kronecker, kamm2000optimal, kilmer2007tplush, nagy1996decomposition, nagy2004reflexivekron}. In contrast, our results are structure-driven and deterministic, and apply to a broad class of block-structured matrices.

\paragraph{Notation and terminology}
Scalars are italicized lowercase ($a$), vectors bold lowercase ($\ba$), index sets italicized
uppercase ($A$), matrices bold uppercase ($\bA$), tensors of order three or higher calligraphic
uppercase ($\tens{A}$), and subspaces of matrices script uppercase ($\mathscr{A}$).  We write
$[n] = \{1,\dots,n\}$.

A matrix that is symmetric across its main antidiagonal is \emph{persymmetric}.  We write
$\bA^{(\gamma,\delta)}$ for the block of $\bA$ in block-position $(\gamma,\delta)$. If $\ell = q$, we say that a matrix is
\emph{block-symmetric} if $\bA^{(\gamma,\delta)} = \bA^{(\delta,\gamma)}$ for all $(\gamma,\delta)$
with $\gamma,\delta\in[\ell]$, and analogously, a matrix is \emph{block-persymmetric} if $\bA^{(\gamma,\delta)} = \bA^{(\ell-\delta+1,\ell-\gamma+1)}$ .

\paragraph{Organization}
Section~\ref{sec:bckgrnd} reviews the Kronecker product, tensor basics, the Tucker decomposition,
and the matrix-to-tensor framework of~\cite{kilmer2022matrixtensordecomp}.
Section~\ref{sec:sets_and_spaces} defines the sets and subspaces used throughout.
Section~\ref{sec:in_and_out} establishes the connection between blockspans and Kronecker rank, and
gives both proofs of the dimension equality.
Section~\ref{sec:bound_r2} derives structure-informed and sparsity-informed upper bounds.
Section~\ref{sec:num_results} presents numerical experiments.
Section~\ref{sec:conclusion} offers conclusions and directions for future work.

\section{Background}
\label{sec:bckgrnd}

\subsection{The Kronecker Product}
\label{subsec:kron}

The Kronecker product of $\bC\in\RR{\ell}{q}$ and $\bD\in\RR{m}{n}$ is the $(\ell m)\times(qn)$
matrix
\begin{equation*}
    \bC \otimes \bD =
    \begin{bmatrix}
        c_{11}\bD & \cdots & c_{1q}\bD\\
        \vdots    & \ddots & \vdots   \\
        c_{\ell 1}\bD & \cdots & c_{\ell q}\bD
    \end{bmatrix}.
\end{equation*}
The following properties will be used
throughout~\cite{Ballard_Kolda_2025,graham2018kronecker,vanloan2000ubiquitous,magnus1979commutation}:
\begin{enumerate}
    \item $(\bC\otimes\bD)^\top = \bC^\top\otimes\bD^\top$.
    \item $\bC\otimes(\bD+\bE) = (\bC\otimes\bD)+(\bC\otimes\bE)$.
    \item $(\bC\otimes\bD)\bx = \mathrm{vec}(\bD\bX\bC^\top)$, where $\bx=\mathrm{vec}(\bX)$ and
          $\mathrm{vec}(\cdot)$ stacks columns.
    \item There exist permutation matrices $\bP$ and $\bQ$ such that
          $$\bP(\bC\otimes\bD)\bQ^\top = \bD\otimes\bC,$$ where $\bP = \bS_{\ell,m}$,
          $\bQ = \bS_{q,n}$, and
          \begin{equation}
          \label{eq:shuffle_matrix}
              \bS_{a,b} =
              \begin{bmatrix}
                  [\bI_s]_{1:b:s,:}\\
                  [\bI_s]_{2:b:s,:}\\
                  \vdots\\
                  [\bI_s]_{b:b:s,:}
              \end{bmatrix}, \quad s = ab.
          \end{equation}
\end{enumerate}

\subsection{Tensor Basics}
\label{sec:tensor_basics}

We use \emph{tensor} to denote a multiway array; unless otherwise stated, all tensors are
third-order.  By fixing all but one index of $\tens{A}$, we obtain \emph{fibers}; by fixing all
but two indices, we obtain two-dimensional \emph{slices}~\cite{kolda2009tensor} (see Figure~\ref{fig:fr_and_lat}).  We write
$\frontal{A}{k} = \tens{A}_{:,:,k}$ for the $k$th frontal slice and
$\lat{A}{j} = \tens{A}_{:,j,:}$ for the $j$th lateral slice.  The operators
$\mathtt{twist}(\cdot)\colon\RRR{m}{p}{1}\to\RRR{m}{1}{p}$ and
$\mathtt{squeeze}(\cdot)\colon\RRR{m}{1}{p}\to\RRR{m}{p}{1}$ reorient frontal slices as lateral
slices and vice versa~\cite{kilmer2013TensorOperator}. See Figure \ref{fig:sq_and_tw} for a visualization of these operators.

\begin{figure}[t]
    \centering
    \includegraphics[width = 0.65\textwidth]{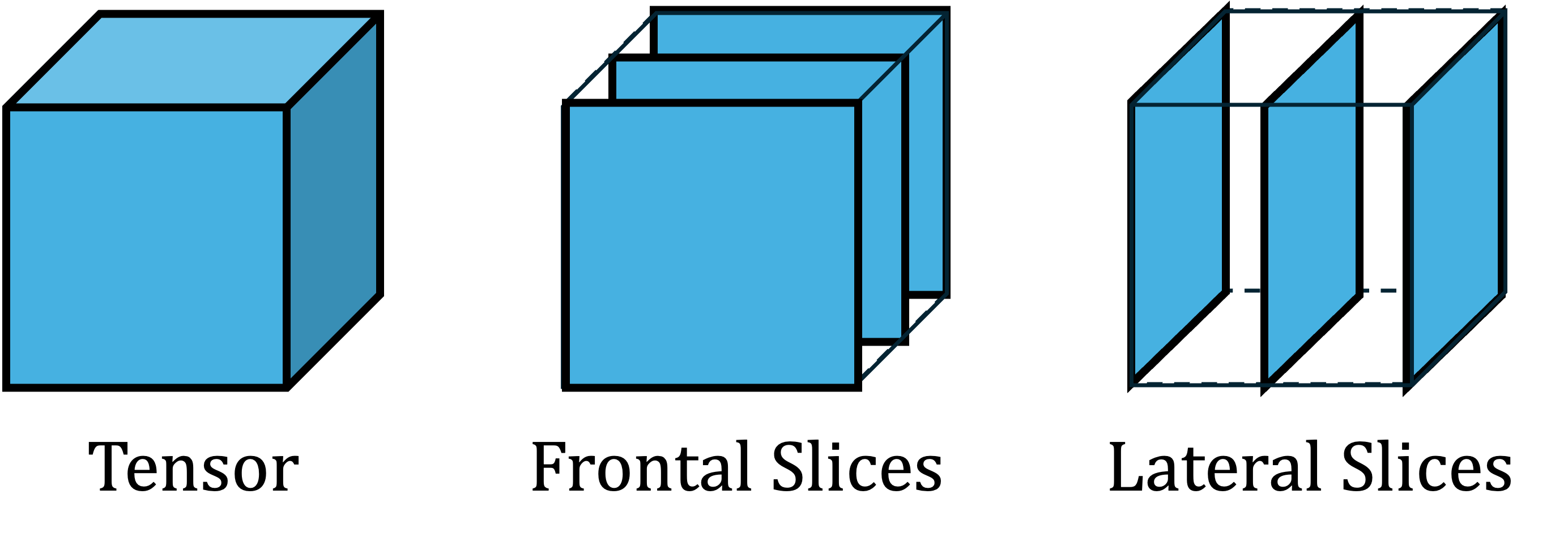}
    \caption{Illustration of a 3-way tensor and two types of slices.}
    \label{fig:fr_and_lat}
\end{figure}

\begin{figure}[t]
    \centering
    \includegraphics[width = \textwidth]{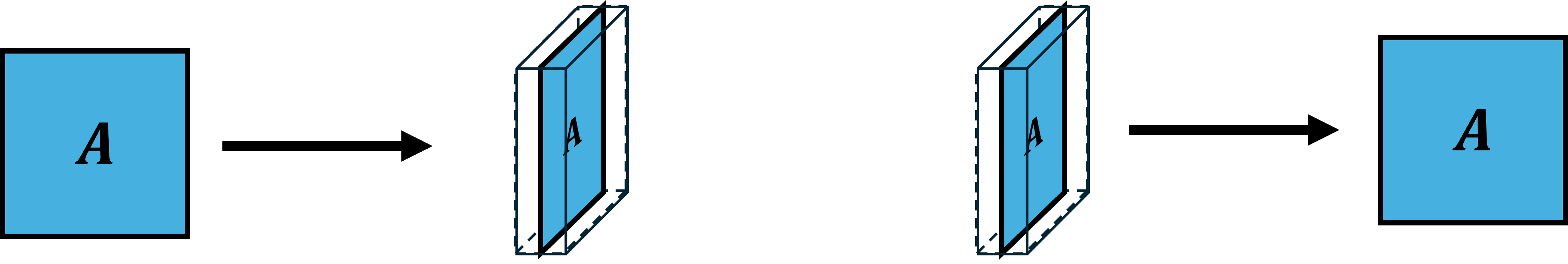}
    \caption{\textbf{Left:} The $\mathtt{twist}(\cdot)$ operator orients a frontal slice as a lateral slice. \textbf{Right:} The $\mathtt{squeeze}(\cdot)$ operator orients a lateral slice as a frontal slice.}
    \label{fig:sq_and_tw}
\end{figure}

The \emph{mode-$i$ unfolding} $\munfold{A}{i}$ is formed by arranging mode-$i$ fibers as columns.
For $\tens{A}\in\RRR{m}{p}{n}$:
\begin{align*}
    \munfold{A}{1} &= \bigl[\frontal{A}{1},\,\frontal{A}{2},\,\dots,\,\frontal{A}{n}\bigr]
                      \in\RR{m}{pn},\\
    \munfold{A}{2} &= \bigl[(\frontal{A}{1})^\top,\,(\frontal{A}{2})^\top,\,\dots,\,
                      (\frontal{A}{n})^\top\bigr]\in\RR{p}{mn},\\
    \munfold{A}{3} &= \bigl[\mathtt{squeeze}(\lat{A}{1}^\top),\,\dots,\,
                      \mathtt{squeeze}(\lat{A}{p}^\top)\bigr]\in\RR{n}{mp}.
\end{align*}
The \emph{multirank} of $\tens{A}$ is $(r_1,r_2,r_3)$ with $r_i=\mathrm{rank}(\munfold{A}{i})$.
\emph{Mode-$i$ multiplication} is defined by
$\tens{C}=\tens{A}\times_i\bB \Leftrightarrow \munfold{C}{i}=\bB\munfold{A}{i}$.

\subsection{The Tucker Decomposition and HOSVD}
\label{sec:tucker}

The Tucker decomposition of $\tens{A}\in\RRR{m}{p}{n}$ is defined by
\begin{equation*}
    \tens{A} = \tens{G}\times_1\bU\times_2\bV\times_3\bW,
\end{equation*}
where $\bU\in\RR{m}{r_1}$, $\bV\in\RR{p}{r_2}$, $\bW\in\RR{n}{r_3}$, and the core tensor
$\tens{G}\in\RRR{r_1}{r_2}{r_3}$~\cite{kolda2009tensor}.  One method for computing this
decomposition is the higher-order SVD (HOSVD)~\cite{delathauwer2000hosvd}, outlined in
Algorithm~\ref{alg:hosvd}.

\begin{algorithm}
    \caption{HOSVD~\protect{\cite{delathauwer2000hosvd}}}
    \label{alg:hosvd}
    \begin{itemize}
        \item Compute $\munfold{A}{i} = \bY_i\bSigma_i\bZ_i^\top$ for $i=1,2,3$.
        \item Set $\bU$, $\bV$, and $\bW$ to be the first $r_1$, $r_2$, $r_3$ columns of
              $\bY_1$, $\bY_2$, and $\bY_3$, respectively, where $r_i = \mathrm{rank}(\munfold{A}{i})$.
        \item Compute $\tens{G} = \tens{A}\times_1\bU^\top\times_2\bV^\top\times_3\bW^\top$.
    \end{itemize}
\end{algorithm}

When the multirank is much smaller than the tensor dimensions, the HOSVD provides an implicitly
compressed representation.  In this work, we focus on the exact HOSVD and, in particular, on
bounds for $r_2$, which will be shown to equal the Kronecker rank.

\subsection{The Matrix-to-Tensor Framework}
\label{sec:mtt}

We now review the construction of \linebreak Kilmer and Saibaba~\cite{kilmer2022matrixtensordecomp}.  Let
$\bA\in\RR{(\ell m)}{(qn)}$ be a block-structured matrix with $p$ distinct blocks
$\bA_k\in\RR{m}{n}$, $k=1,\dots,p$.  Define the \emph{location-tally matrices}
$\bE_k\in\RR{\ell}{q}$ by
\begin{equation}
\label{eq:E}
    [\bE_k]_{ij} =
    \begin{cases}
        1/\sqrt{\eta_k} & \text{if block } \bA_k \text{ occurs at block position }(i,j),\\
        0 & \text{otherwise,}
    \end{cases}
\end{equation}
where $\eta_k$ is the number of occurrences of $\bA_k$.  Then,
\begin{equation}
\label{eq:A_struct_kron}
    \bA = \sum_{k=1}^p \bE_k \otimes \sqrt{\eta_k}\bA_k.
\end{equation}
A visual representation of Equation \eqref{eq:A_struct_kron} is given in Figure \ref{fig:aid_A_and_E} for a block-symmetric block-Toeplitz matrix $\bA$ with three distinct blocks.
\begin{figure}[t]
    \centering
    \includegraphics[width = \textwidth]{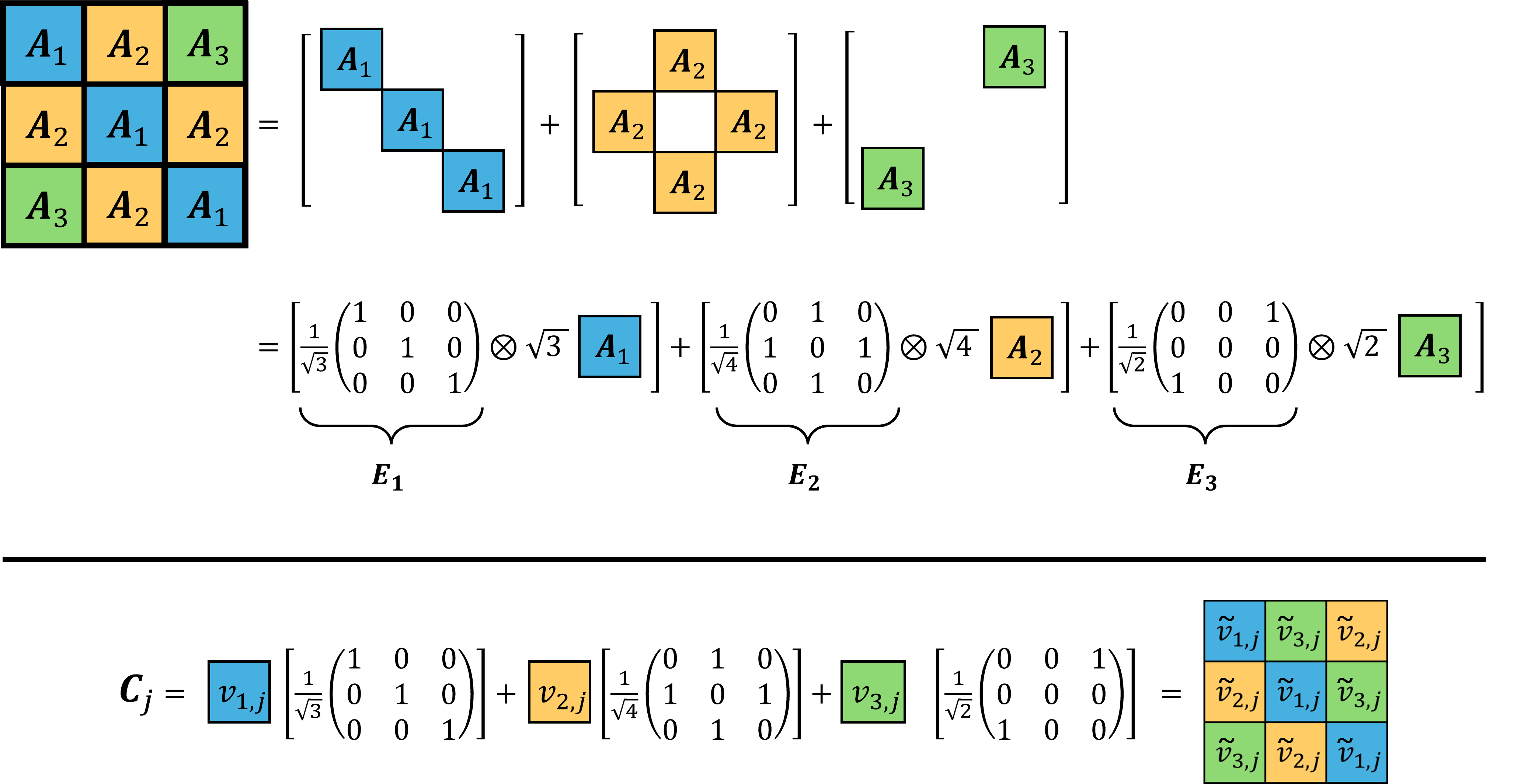}
    \caption{\textbf{Top:} Visual representation of Equation \eqref{eq:A_struct_kron}. On the left is a block-symmetric block-Toeplitz matrix $\bA$ with three distinct blocks. On the right, $\bA$ is written as a sum of Kronecker products of its blocks and the matrices $\bE_k$, as defined in Equation \eqref{eq:E}. \textbf{Bottom:} Visual representation of a matrix $\bC_j$ described in Theorem \ref{thm:r2=dimC} for the same structure. The elements of $\bC_j$ are arranged in the same block pattern as $\bA$. In the right-most matrix, the scaled coefficient $\tilde{v}_{kj} := v_{kj}/\sqrt{\eta_k}$ arises from substituting the definition of $\bE_k$ into Equation \eqref{eq:Tucker_TTM}.}
    \label{fig:aid_A_and_E}
\end{figure}

The \emph{matrix-to-tensor map} $\mathcal{T}\colon\RR{(\ell m)}{(qn)}\to\RRR{m}{p}{n}$ is defined
by setting the $k$th lateral slice of $\tens{A} := \mathcal{T}[\bA]$ to
\begin{equation}
\label{eq:mtt_map}
    \lat{A}{k} = \mathtt{twist}(\sqrt{\eta_k}\bA_k),\quad k=1,\dots,p.
\end{equation}
The inverse \emph{tensor-to-matrix map} $\mathcal{M}\colon\RRR{m}{p}{n}\to\RR{(\ell m)}{(qn)}$
is
\begin{equation*}
    \mathcal{M}[\tens{A}] = \sum_{k=1}^p \bE_k\otimes\mathtt{squeeze}(\lat{A}{k}).
\end{equation*}
If $\tens{A}$ admits the Tucker decomposition
$\tens{A}=\tens{G}\times_1\bU\times_2\bV\times_3\bW$ with
$\tens{G}\in\RRR{r_1}{r_2}{r_3}$, then
\begin{equation}
\label{eq:Tucker_TTM}
    \bA = \mathcal{M}[\tens{A}] = \sum_{j=1}^{r_2}\bC_j\otimes\bD_j,
\end{equation}
where $\bC_j = \sum_{k=1}^p v_{kj}\bE_k$, $\bD_j = \bU\bG_j\bW^\top$,
$v_{kj} = [\bV]_{kj}$, and $\bG_j = \mathtt{squeeze}(\lat{G}{j})$.  The integer $r_2$ is thus
the number of terms in the Kronecker sum produced by this framework.  Figure~\ref{fig:MTT_TTM_diagram}
illustrates the full pipeline.

\begin{figure}[t]
    \centering
    \includegraphics[width=\textwidth]{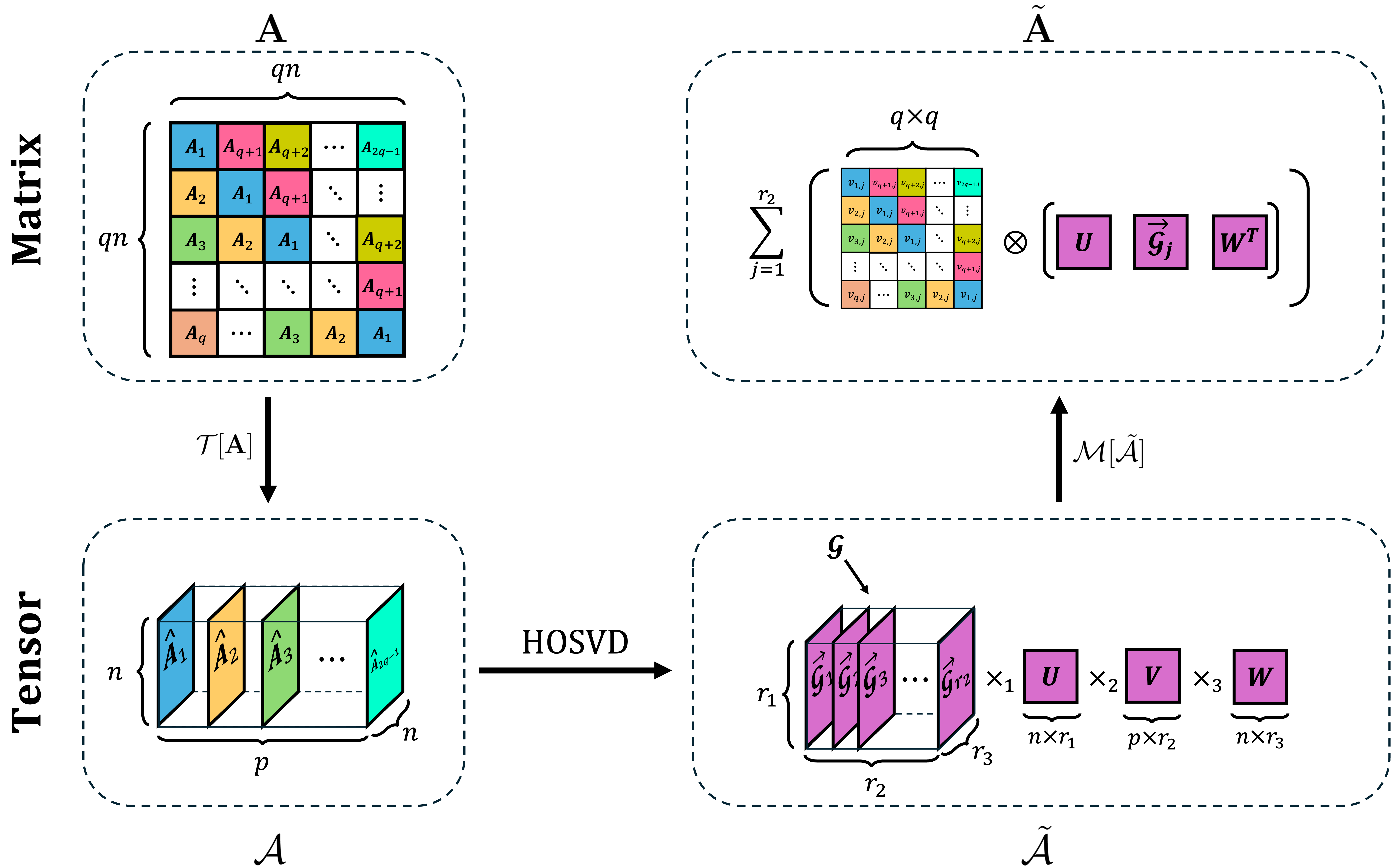}
    \caption{Pipeline for rewriting $\bA$ as a sum of Kronecker products via the MTT map,
             HOSVD, and TTM map of~\protect{\cite{kilmer2022matrixtensordecomp}}.}
    \label{fig:MTT_TTM_diagram}
\end{figure}

\section{Sets, Subspaces, and the Permuted Matrix}
\label{sec:sets_and_spaces}

\subsection{The Permuted Matrix}
\label{subsec:permuted}

By property~4 of the Kronecker product, there exist permutation matrices
$\bP\in\RR{(\ell m)}{(\ell m)}$ and $\bQ\in\RR{(qn)}{(qn)}$ with $\bP=\bS_{\ell,m}$ and
$\bQ=\bS_{q,n}$ such that
\begin{equation*}
    \bP(\bE_k\otimes\sqrt{\eta_k}\bA_k)\bQ^\top = \sqrt{\eta_k}\bA_k\otimes\bE_k.
\end{equation*}
Define $\bB = \bP\bA\bQ^\top \in\RR{(m\ell)}{(nq)}$.  This matrix is block-structured with an
$m\times n$ grid of blocks each of size $\ell\times q$, and satisfies
\begin{equation}
\label{eq:matABF}
    \bA = \bP^\top\bB\bQ = \sum_{\kappa=1}^{\rho}\sqrt{\xi_\kappa}\bB_\kappa\otimes\bF_\kappa,
\end{equation}
where $\bB_\kappa \in \mathbb{R}^{\ell \times q}$ are the $\rho$ distinct blocks of $\bB$, $\xi_\kappa$ is the multiplicity of
$\bB_\kappa$, and $\bF_\kappa\in\RR{m}{n}$ is the corresponding location-tally matrix.  We use
$k,p$ for counts associated with $\bA$ and $\kappa,\rho$ for those associated with $\bB$.
\begin{figure}[t]
    \centering
    \includegraphics[width = \textwidth]{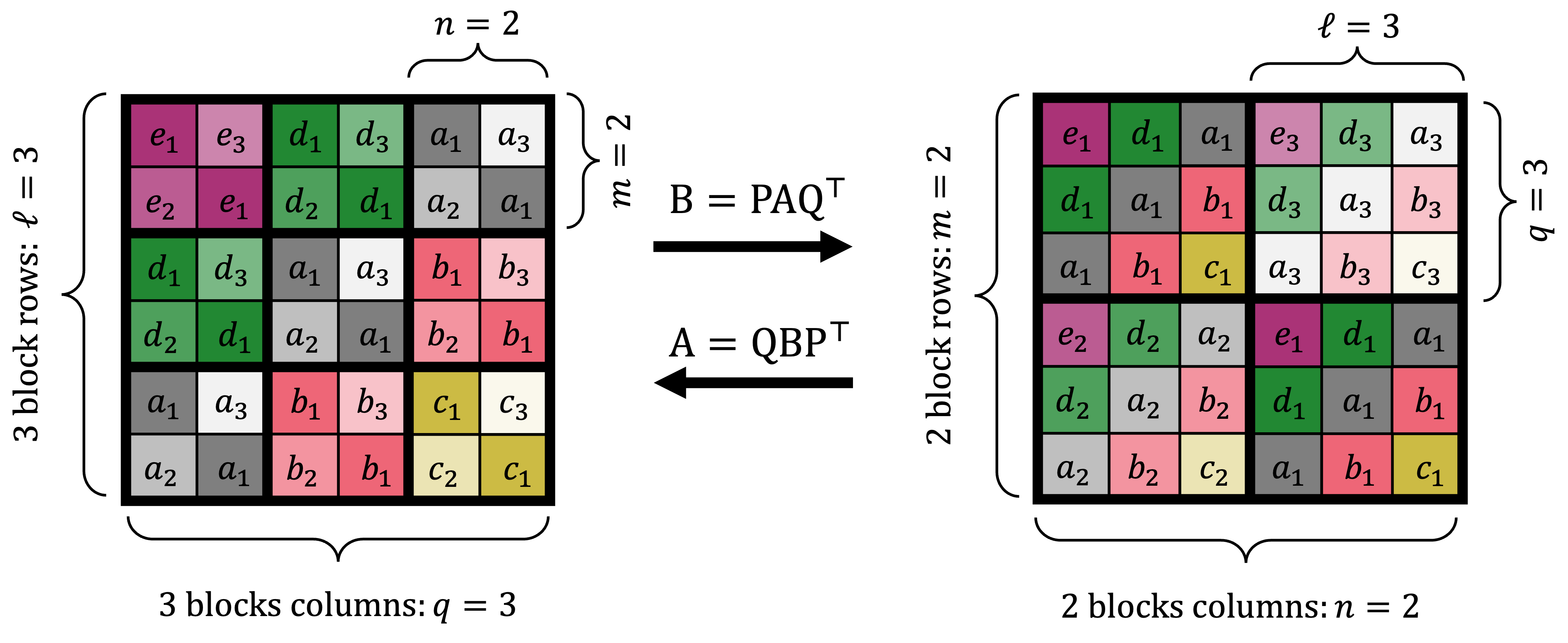}
   \caption{On the left is a block-Hankel matrix $\bA$ with Toeplitz blocks. After permuting $\bA$ with $\bP$ and $\bQ$, we obtain a matrix $\bB$ that is block-Toeplitz with Hankel blocks.}
    \label{fig:permute}
\end{figure}

\begin{figure}[t]
    \centering
    \begin{subfigure}{0.85\linewidth}
    \includegraphics[width = 0.85\textwidth]{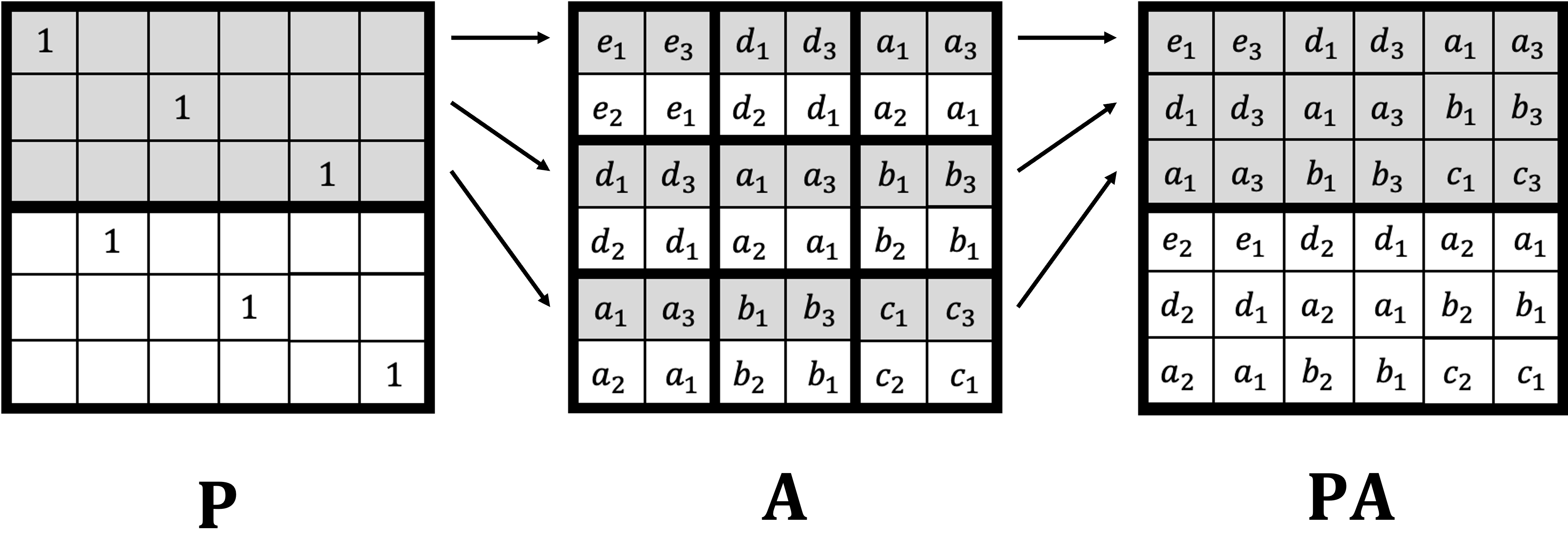}
    \end{subfigure}
    
    \vspace{3mm}
    
    \begin{subfigure}{0.85\linewidth}
    \includegraphics[width = 0.85\textwidth]{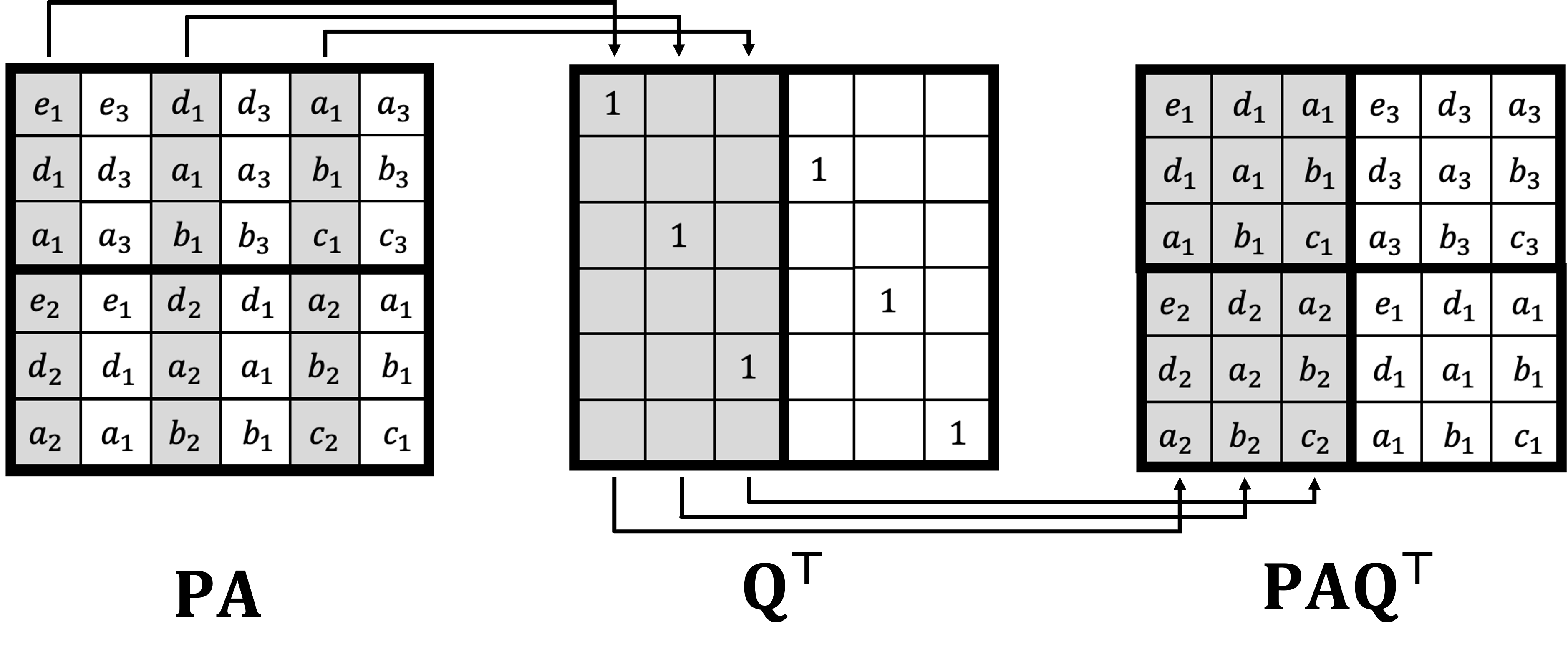}
    \end{subfigure}
    
    \caption{The effects of the permutation matrix $\bP$ on the rows of $\bA$ (Figure \ref{fig:permute}). The $\alpha$th block-row of $\bP = \bS_{\ell,m}$ selects row $\alpha$ from each block-row of $\bA$ and consolidates those rows into block-row $\alpha$ of $\bP\bA$. The matrix $\bQ^\top$ acts analogously on columns.}
    \label{fig:shuffle}
\end{figure}

\begin{example}
    Consider Figure \ref{fig:permute}. On the left, it shows a block-Hankel matrix~$\bA$. The matrix has three block-rows and three block-columns, thus $\ell = q = 3$. Each block of $\bA$ is a $2\times 2$ Toeplitz matrix, thus $m=n=2$. By applying appropriate permutation matrices $\bP$ and $\bQ^\top$ to $\bA$, we get the matrix $\bB$ shown on the right. This matrix is a block-Toeplitz matrix with two block-rows and two block-columns, where each block is a $3\times 3$ Hankel matrix. Thus, the permutation matrices have swapped the outer and inner structures of $\bA$.
\end{example}

Let $(\alpha,\beta)\in[m]\times[n]$ and $(\gamma,\delta)\in[\ell]\times[q]$. Let $\bA^{(\gamma, \delta)}$ denote the block of $\bA$ in block-location $(\gamma,\delta)$, and $\bB^{(\alpha, \beta)}$ the block of $\bB$ in block-location $(\alpha,\beta)$. The following lemma makes the element-wise relationship between $\bA$ and $\bB$ precise.

\begin{lemma}\label{lemma:permutation}
Let $\bA\in\RR{(\ell m)}{(qn)}$ have $\ell q$ blocks of size $m\times n$, and let $\bB = \bP\bA\bQ^\top$ as above. Then,
\begin{equation}
\label{eq:element_swap}
    \bB^{(\alpha,\beta)}_{\gamma\delta} = \bA^{(\gamma,\delta)}_{\alpha\beta}
    \quad\text{for all }(\alpha,\beta)\in[m]\times[n],\;(\gamma,\delta)\in[\ell]\times[q].
\end{equation}
\end{lemma}

\begin{proof}
The global row and column indices satisfy
\begin{equation}
\label{eq:index_translation}
    \bA^{(\gamma,\delta)}_{\alpha\beta} = \bA_{(\gamma-1)m+\alpha,\;(\delta-1)n+\beta},
    \qquad
    \bB^{(\alpha,\beta)}_{\gamma\delta} = \bB_{(\alpha-1)\ell+\gamma,\;(\beta-1)q+\delta}.
\end{equation}
It therefore suffices to show that
\begin{equation}
\label{eq:goal}
    [\bP\bA\bQ^\top]_{(\alpha-1)\ell+\gamma,\;(\beta-1)q+\delta}
    = \bA_{(\gamma-1)m+\alpha,\;(\delta-1)n+\beta}.
\end{equation}

\medskip\noindent\textit{Action of $\bP$ on rows.}
The shuffle matrix $\bP = \bS_{\ell,m}\in\RR{(\ell m)}{(\ell m)}$ has block structure: its $\alpha$th
block-row (for $\alpha\in[m]$) consists of $\ell$ rows, and its $\gamma$th row within that block
is the standard basis vector $\be_{(\gamma-1)m+\alpha}^\top$ of length $\ell m$.  Consequently,
row $(\alpha-1)\ell+\gamma$ of $\bP\bA$ is row $(\gamma-1)m+\alpha$ of $\bA$:
\begin{equation}
\label{eq:P_action}
    [\bP\bA]_{(\alpha-1)\ell+\gamma,\;j} = \bA_{(\gamma-1)m+\alpha,\;j}
    \quad\text{for all }j.
\end{equation}

\medskip\noindent\textit{Action of $\bQ^\top$ on columns.}
By the same reasoning applied to columns, $\bQ = \bS_{q,n}\in\RR{(qn)}{(qn)}$ satisfies
\begin{equation}
\label{eq:Q_action}
    [\bP\bA\bQ^\top]_{i,\;(\beta-1)q+\delta} = [\bP\bA]_{i,\;(\delta-1)n+\beta}
    \quad\text{for all }i.
\end{equation}
Combining \eqref{eq:P_action} and \eqref{eq:Q_action} with $i=(\alpha-1)\ell+\gamma$ gives
\eqref{eq:goal}, which is equivalent to~\eqref{eq:element_swap}.
\end{proof}

In the example of Figure \ref{fig:permute}, one may verify directly that $\bA^{(2,1)}_{2,2} = \bB^{(2,2)}_{2,1} = d_1$, consistent with \eqref{eq:element_swap}.

\subsection{Sets and Subspaces}
\label{subsec:sets_and_spaces}

Let $A = \{\bA_k\}_{k=1}^p$ be the set of distinct blocks of $\bA$, let
$E = \{\bE_k\}_{k=1}^p$ be the associated location-tally matrices, let
$B = \{\bB_\kappa\}_{\kappa=1}^\rho$ be the distinct blocks of $\bB$, and let
$F = \{\bF_\kappa\}_{\kappa=1}^\rho$ be the corresponding tally matrices.

\begin{definition}
\label{def:subspaces}
Define the following four subspaces:
\begin{itemize}
    \item $\mathscr{A} = \mathrm{span}(A)$: the \emph{inner blockspan} of $\bA$;
    \item $\mathscr{B} = \mathrm{span}(B)$: the \emph{outer blockspan} of $\bA$; 
    \item $\mathscr{E} = \mathrm{span}(E)$: the \emph{outer structure space} of $\bA$;
    \item $\mathscr{F} = \mathrm{span}(F)$: the \emph{inner structure space} of $\bA$.
\end{itemize}
\end{definition}

\begin{proposition}
\label{lemma:subspace_containment}
With the notation above, $\mathscr{B}\subseteq\mathscr{E}$ and $\mathscr{A}\subseteq\mathscr{F}$.
\end{proposition}

\begin{proof}
By Lemma~\ref{lemma:permutation}, each $\bB_\kappa\in B$ has entries
$b_\kappa^k := [\bB_\kappa]_{\gamma\delta}$ (for any $(\gamma,\delta)$ such that
$\bA^{(\gamma,\delta)}=\bA_k$) placed in the same pattern as $\bA_k$ occurs in $\bA$.  Hence
$\bB_\kappa = \sum_{k=1}^p b_\kappa^k(\sqrt{\eta_k}\bE_k)$, so every generator of $\mathscr{B}$
lies in $\mathscr{E}$, giving $\mathscr{B}\subseteq\mathscr{E}$.  The containment
$\mathscr{A}\subseteq\mathscr{F}$ follows by the same argument applied to $\bB$.
\end{proof}

\section{Blockspans and Kronecker Rank}
\label{sec:in_and_out}
In this section, we state and prove our main theoretical results connecting the blockspans and Kronecker rank.

\subsection{Kronecker Rank Equals Dimension of the Inner Blockspan}

\begin{theorem}
\label{thm:r2=dimC}
Let $\bA$ be a block-structured matrix with $\ell\times q$ blocks of size $m\times n$, and let
$\tens{A} = \mathcal{T}[\bA]$ be the associated rank-$(r_1,r_2,r_3)$ tensor.  Then,
$r_2 = \dim(\mathscr{A})$.
\end{theorem}

\begin{proof}
Since $\lat{A}{k} = \mathtt{twist}(\sqrt{\eta_k}\bA_k)$, the $i$th column of $\sqrt{\eta_k}\bA_k$
is the $k$th column of $\frontal{A}{i}$, hence the $k$th row of $(\frontal{A}{i})^\top$.  Thus,
\begin{equation}
\label{eq:A2_row_form}
    \munfold{A}{2} =
    \begin{bmatrix}
        \ba_1^\top \\ \ba_2^\top \\ \vdots \\ \ba_p^\top
    \end{bmatrix}\in\RR{p}{mn},
\end{equation}
where $\ba_k = \mathrm{vec}(\sqrt{\eta_k}\bA_k)$. See Figure \ref{fig:A_A_A2} for a visualization. Therefore,
$$r_2 = \mathrm{rank}(\munfold{A}{2}) = \dim(\mathrm{span}\{\ba_k\}) =
\dim(\mathrm{span}\{\bA_k\}) = \dim(\mathscr{A}).$$
This concludes the proof.
\end{proof}

\begin{figure}[t]
    \centering
    \includegraphics[width = \textwidth]{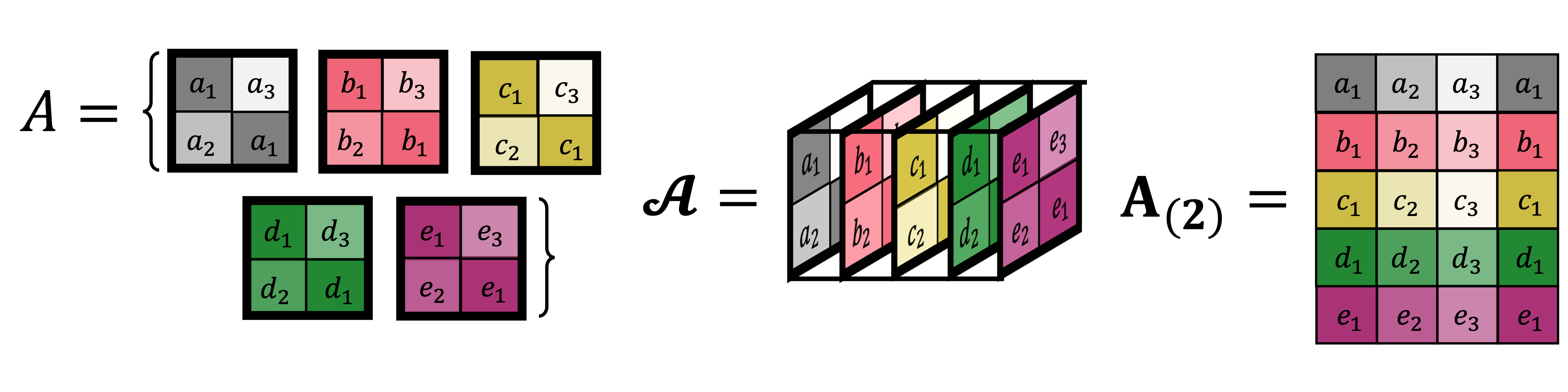}
   \caption{\textbf{Left:} The unique blocks of the matrix $\bA$ shown in Figure \ref{fig:permute}. \textbf{Center:} The tensor $\tens{A}$ mapped from $\bA$. \textbf{Right:} The second mode unfolding $\munfold{A}{2}$ of $\tens{A}$.}
    \label{fig:A_A_A2}
\end{figure}

\begin{remark}
\label{rem:reshuffled}
The matrix $\munfold{A}{2}$ in~\eqref{eq:A2_row_form} is the reshuffled matrix
$\mathscr{R}(\bA)$ of Van Loan and Pitsianis~\protect{\cite{vanloan2000ubiquitous, loan1992approximation}}
with (i) rows weighted by $\sqrt{\eta_k}$ and (ii) duplicate rows removed.  Both matrices, therefore, have the same row space and the same rank.  This is consistent with the
equivalence between the two frameworks established in prior
work~\protect{\cite{fuller2026tensor,kilmer2022matrixtensordecomp}}.
\end{remark}

\subsection{Equality of Inner and Outer Blockspan Dimensions: First Proof}\label{sec:first_proof}

Because $\bA = \bP^\top\bB\bQ$, we can write $\bA$ as a sum of $r_2$ or $\bar r_2 :=
\dim(\mathscr{B})$ Kronecker products.  The following theorem shows that these two quantities coincide.

\begin{theorem}
\label{thm:dimA=dimB}
Let $\bA$ be a block-structured matrix and let $\bB = \bP\bA\bQ^\top$ as above.  Then,
$\dim(\mathscr{A}) = \dim(\mathscr{B})$.
\end{theorem}

\begin{proof}
Let $\bar r_2 = \dim(\mathscr{B})$ and let $\{\bT_1,\dots,\bT_{\bar r_2}\}$ be a basis for $\mathscr{B}$.  For each $\bB_\kappa\in B$, write $\bB_\kappa = \sum_{j=1}^{\bar r_2}\phi_j^\kappa\bT_j$.  Substituting into~\eqref{eq:matABF},
\begin{equation*}
    \bA = \sum_{\kappa=1}^\rho\sqrt{\xi_\kappa}
          \Bigl(\sum_{j=1}^{\bar r_2}\phi_j^\kappa\bT_j\Bigr)\otimes\bF_\kappa
        = \sum_{j=1}^{\bar r_2}\bT_j\otimes\hat{\bF}_j,
\end{equation*}
where $\hat{\bF}_j = \sum_{\kappa=1}^\rho\phi_j^\kappa\sqrt{\xi_\kappa}\bF_\kappa$.  Every block
$\bA^{(\gamma,\delta)} = \sum_{j=1}^{\bar r_2}[\bT_j]_{\gamma\delta}\,\hat{\bF}_j$ lies in
\linebreak $\mathrm{span}\{\hat{\bF}_j\}_{j=1}^{\bar r_2}$, so $\dim(\mathscr{A})\le\bar r_2 = \dim(\mathscr{B})$.

For the reverse inequality, observe that $\bB$ is itself a block-structured matrix: it has an $m\times n$
block grid with blocks of size $\ell\times q$.  By Lemma~\ref{lemma:permutation}, the $(\gamma,\delta)$
entry of block $\bB^{(\alpha,\beta)}$ equals $[\bA^{(\gamma,\delta)}]_{\alpha\beta}$; that is, each
block of $\bB$ is determined by a fixed entry position $(\alpha,\beta)$ within the blocks of $\bA$.
Two positions $(\alpha,\beta)$ and $(\alpha',\beta')$ produce the same block of $\bB$ if and only if
the corresponding entries agree across all blocks of $\bA$, i.e.,
$[\bA^{(\gamma,\delta)}]_{\alpha\beta} = [\bA^{(\gamma,\delta)}]_{\alpha'\beta'}$ for all
$(\gamma,\delta)$.  The inner blockspan of $\bB$ is $\mathscr{B}$, and its outer blockspan
is $\mathscr{A}$. Applying the preceding argument to $\bB$ in place of
$\bA$ gives $\dim(\mathscr{B})\le\dim(\mathscr{A})$, completing the proof.
\end{proof}

\subsection{Equality of Inner and Outer Blockspan Dimensions: Second Proof via Isomorphism}
\label{sec:second_proof}

The first proof establishes the equality by a dimension argument.  Here we give a second proof by showing that the outer blockspan $\mathscr{B}$ is isomorphic to the column space of
$\munfold{A}{2}$, whose dimension is $r_2 = \dim(\mathscr{A})$ by Theorem~\ref{thm:r2=dimC}.

Recall from the proof of Proposition~\ref{lemma:subspace_containment} that each $\bB_\kappa\in B$
satisfies $\bB_\kappa = \sum_{k=1}^p b_\kappa^k(\sqrt{\eta_k}\bE_k)$.  Define two linear maps:
\begin{equation}
\label{eq:vec_mat_ops}
    \mathtt{vec}_{\mathscr{E}}(\bB_\kappa)
    = \begin{bmatrix}\sqrt{\eta_1}b_\kappa^1\\\vdots\\\sqrt{\eta_p}b_\kappa^p\end{bmatrix}\in\mathbb{R}^p \quad  \mbox{ and } \quad
    \mathtt{mat}_{\mathscr{E}}(\bv) = \sum_{k=1}^p v_k\bE_k,\quad\bv\in\mathbb{R}^p.
\end{equation}
Both maps extend linearly to their respective domains.

\begin{lemma}
\label{lemma:B_iso_col_A2}
$\mathtt{vec}_{\mathscr{E}}(\cdot)$ is a linear bijection from $\mathscr{B}$ to
$\mathrm{colspace}(\munfold{A}{2})$, with inverse $\mathtt{mat}_{\mathscr{E}}(\cdot)$.  Hence
$\mathscr{B}\cong\mathrm{colspace}(\munfold{A}{2})$.
\end{lemma}

\begin{proof}
\textit{Step 1: $\mathtt{vec}_{\mathscr{E}}$ maps $\mathscr{B}$ into $\mathrm{colspace}(\munfold{A}{2})$.}
Every $\bT\in\mathscr{B}$ writes as $\bT = \sum_\kappa \zeta_\kappa\bB_\kappa$, so by linearity
\[ \mathtt{vec}_{\mathscr{E}}(\bT)
= \sum_{\kappa=1}^{\rho}\zeta_\kappa
\begin{bmatrix}\sqrt{\eta_1}b_\kappa^1 \\ \vdots\\ \sqrt{\eta_p}b_\kappa^p\end{bmatrix}.
\]
By Equation~\eqref{eq:A2_row_form}, column $(\alpha,\beta)$ of $\munfold{A}{2}$ is the vector
$([\bA_1]_{\alpha\beta}\sqrt{\eta_1},\dots,[\bA_p]_{\alpha\beta}\sqrt{\eta_p})^\top$.
By Lemma~\ref{lemma:permutation}, $[\bB_\kappa]_{\gamma\delta} = [\bA^{(\gamma,\delta)}]_{\alpha\beta}$
for any $(\alpha,\beta)$ with $\bB^{(\alpha,\beta)}=\bB_\kappa$, so \linebreak
$(\sqrt{\eta_1}b_\kappa^1,\dots,\sqrt{\eta_p}b_\kappa^p)^\top$ is exactly column $(\alpha,\beta)$
of $\munfold{A}{2}$.  Two distinct blocks $\bB_\kappa\ne\bB_{\kappa'}$ produce distinct column vectors (if they agreed entrywise, they would not be distinct blocks by definition); multiple positions $(\alpha,\beta)$ with the same block $\bB_\kappa$ produce identical columns, but the span is unchanged.  Hence
$\mathtt{vec}_{\mathscr{E}}(\bT)\in\mathrm{colspace}(\munfold{A}{2})$ and thus $\mathtt{vec}_{\mathscr{E}}(\mathscr{B})\subseteq\mathrm{colspace}(\munfold{A}{2})$.

\textit{Step 2: $\mathtt{mat}_{\mathscr{E}}$ maps $\mathrm{colspace}(\munfold{A}{2})$ into $\mathscr{B}$.}
By Equation~\eqref{eq:A2_row_form}, column $(i,j)$ of $\munfold{A}{2}$ is
$\bc_{ij} = ([\bA_1]_{ij}\sqrt{\eta_1},\dots,[\bA_p]_{ij}\sqrt{\eta_p})^\top$.
Applying $\mathtt{mat}_{\mathscr{E}}$ gives
\[
\mathtt{mat}_{\mathscr{E}}(\bc_{ij}) = \sum_{k=1}^p [\bA_k]_{ij}\sqrt{\eta_k}\,\bE_k.
\]
By Lemma~\ref{lemma:permutation}, the $(\gamma,\delta)$ entry of the block $\bB^{(i,j)}$ of $\bB$
equals $[\bA^{(\gamma,\delta)}]_{ij}$.  Summing over $k$ with the location-tally weights shows
that $\mathtt{mat}_{\mathscr{E}}(\bc_{ij})$ is precisely the matrix $\bB^{(i,j)}$,
which belongs to $B\subset\mathscr{B}$.  Since $\mathrm{colspace}(\munfold{A}{2})$ is spanned by
$\{\bc_{ij}\}$, linearity gives $\mathtt{mat}_{\mathscr{E}}(\mathrm{colspace}(\munfold{A}{2}))\subseteq\mathscr{B}$. 

\textit{Step 3: The maps are mutual inverses.}
A direct substitution verifies
$$\mathtt{mat}_{\mathscr{E}}(\mathtt{vec}_{\mathscr{E}}(\bB_\kappa))=\bB_\kappa \mbox{ and } 
\mathtt{vec}_{\mathscr{E}}(\mathtt{mat}_{\mathscr{E}}(\bv))=\bv,$$
so they are inverse bijections, establishing $\mathscr{B}\cong\mathrm{colspace}(\munfold{A}{2})$.
\end{proof}

Theorem~\ref{thm:dimA=dimB} follows immediately:
$\dim(\mathscr{B}) = \dim(\mathrm{colspace}(\munfold{A}{2})) = r_2 = \dim(\mathscr{A})$.

\section{Structure-Informed and Sparsity-Informed Bounds}
\label{sec:bound_r2}

\subsection{Structural Bounds}
\label{subsec:structural_bounds}

From Theorem~\ref{thm:r2=dimC}, the Kronecker rank $r_2 = \dim(\mathscr{A})$.  When the blocks
of $\bA$ are not known explicitly, one can bound $\dim(\mathscr{A})$ by working with larger
spaces defined by the structural type of the blocks.

Let $\mathscr{S}$ denote the subspace of matrices whose form is determined by the structural type
of the blocks of $\bA$ (e.g., the space of $m\times n$ Toeplitz matrices if each block is
Toeplitz), and let $\mathscr{T}$ denote the analogous space for the outer block structure (i.e.,
the structural type of the blocks of $\bB$).  By construction,
$\mathscr{A}\subseteq\mathscr{S}\subseteq\mathscr{F}$ and $\mathscr{B}\subseteq\mathscr{T}
\subseteq\mathscr{E}$.

\begin{corollary}
\label{cor:nested}
With $\mathscr{S}$ and $\mathscr{T}$ defined as above,
\begin{equation*}
    r_2 \le \min\bigl\{\dim(\mathscr{S}),\,\dim(\mathscr{T})\bigr\}.
\end{equation*}
\end{corollary}

Figure~\ref{fig:diag_of_bounds} illustrates these containments.

\begin{figure}[t]
    \centering
    \includegraphics[width=0.7\textwidth]{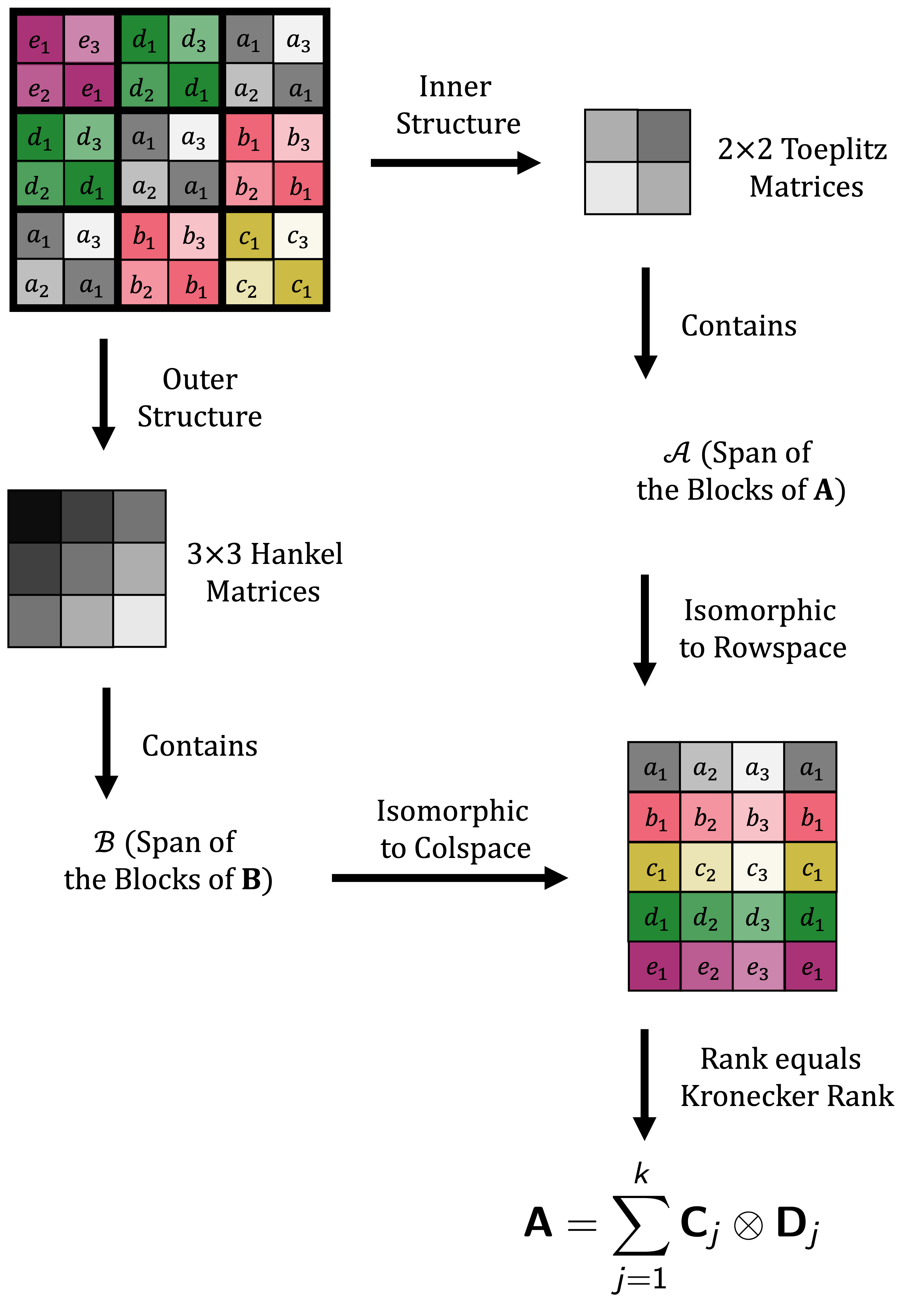}
    \caption{The nested containments $\mathscr{A}\subseteq\mathscr{S}\subseteq\mathscr{F}$ and
    $\mathscr{B}\subseteq\mathscr{T}\subseteq\mathscr{E}$ connecting structural spaces to the
    Kronecker rank.  Dimensions of $\mathscr{S}$ and $\mathscr{T}$ provide computable upper
    bounds on $r_2$.}
    \label{fig:diag_of_bounds}
\end{figure}

\begin{remark}
\label{rem:E_vs_T}
The distinction between $\mathscr{T}$ and $\mathscr{E}$ is important.  The set $E$ is built from
the \emph{distinct} blocks of $\bA$ in the elementwise sense, whereas $\mathscr{T}$ is defined
by linear structure.  For example, if $\bA$ has a symmetric block-Toeplitz-plus-Hankel outer
structure then $\mathscr{T}$ is the space of $q\times q$ symmetric Toeplitz-plus-Hankel matrices
with $\dim(\mathscr{T})=2q-2$, while $\bA$ has $(q^2/4)+(q/2)$ distinct blocks giving
$\dim(\mathscr{E})=(q^2/4)+(q/2)$ (see Appendix~\ref{app:t+h_worked} for the derivation of
this formula in the case $q=5$).  For $q\ge 5$ we have $\mathscr{T}\subsetneq\mathscr{E}$, and
$\dim(\mathscr{T})$ gives a strictly tighter bound.
\end{remark}

\begin{remark}
\label{rem:A_vs_S}
Similarly, $\dim(\mathscr{A})\le\dim(\mathscr{S})$ and the inequality can be strict.  For
instance, if all blocks of $\bA$ are scalar multiples of a single matrix, then
$\dim(\mathscr{A})=1$ regardless of what structural type the blocks possess.
\end{remark}

\paragraph{Table of bounds}
Table~\ref{tab:estimate} records $\dim(\mathscr{T})$ and $\dim(\mathscr{S})$ for three
representative matrix families, for two choices of $q$ (outer block grid size) and $n$ (block
size).  The tighter of the two bounds is highlighted.  All formulas assume that $q$ and $n$ are even.

\begin{table}[t]
    \centering
    \begin{tabular}{c c c c c c c c}
    \cmidrule[1pt]{1-8}
    && \textbf{Inner struct.} & \textbf{Outer struct.}
       && $\dim(\mathscr{T})$ & $\dim(\mathscr{S})$ & $\min$\\
    \cmidrule{1-1}\cmidrule{3-4}\cmidrule{6-8}
    \multirow{3}*{$\begin{matrix}\mathbf{q=32}\\ \mathbf{n=64}\end{matrix}$}
    && Sym.\ T+H     & Nonsym.\ T+H && \textbf{62}  & 252 & 62\\
    && Nonsym.\ T+H  & Sym.\ Toeplitz && 124 & \textbf{64} & 64\\
    && Nonsym.\ T+H  & Tridiagonal   && \textbf{124} & 190 & 124\\
    \cmidrule{1-1}\cmidrule{3-4}\cmidrule{6-8}
    \multirow{3}*{$\begin{matrix}\mathbf{q=64}\\ \mathbf{n=32}\end{matrix}$}
    && Sym.\ T+H     & Nonsym.\ T+H && 126 & \textbf{124} & 124\\
    && Nonsym.\ T+H  & Sym.\ Toeplitz && 252 & \textbf{32} & 32\\
    && Nonsym.\ T+H  & Tridiagonal   && 252 & \textbf{94} & \textbf{94}\\
    \bottomrule
    \end{tabular}
    \caption{Structural upper bounds on $r_2$ for three matrix families and two block-size
    combinations.  ``Sym.\ T+H'' = symmetric Toeplitz plus persymmetric Hankel;
    ``Nonsym.\ T+H'' = Toeplitz-plus-Hankel without symmetry conditions.
    The overall bound is $\min\{\dim(\mathscr{T}),\dim(\mathscr{S})\}$.}
    \label{tab:estimate}
\end{table}

\subsection{Sparsity-Informed Bounds}
\label{subsec:sparsity}

When the blocks of $\bA$ are sparse, the rank of $\munfold{A}{2}$ is bounded by the number of
positions that are nonzero in \emph{at least one} block. This is formalized in the following theorem.

\begin{theorem}
\label{thm:sparse}
Let $\bA\in\RR{(\ell m)}{(qn)}$ be a block-structured matrix. For each distinct block $\bA_k$, define the nonzero index set
$C_k := \{(i,j)\mid[\bA_k]_{ij}\ne 0\}$. Then,
\begin{equation*}
    r_2 \le \mathtt{card}\!\Bigl(\bigcup_{k=1}^p C_k\Bigr).
\end{equation*}
\end{theorem}

\begin{proof}
Let $c = \mathtt{card}(\bigcup_k C_k)$.  By Equation~\eqref{eq:A2_row_form}, each row of $\munfold{A}{2}$
is $\ba_k = \mathrm{vec}(\sqrt{\eta_k}\bA_k)$.  A column of $\munfold{A}{2}$, indexed by position
$(i,j)$, is zero if $[\bA_k]_{ij}=0$ for all $k$, i.e., if $(i,j)\notin\bigcup_k C_k$.  Hence
$\munfold{A}{2}$ has at most $c$ nonzero columns, giving $r_2 \le c$.
\end{proof}

When all blocks share the same sparsity pattern, $\bigcup_k C_k = C_k$ for any $k$, so
$r_2\le\mathtt{nnz}(\bA_k)$.  For banded blocks, the following corollary applies immediately.

\begin{corollary}
\label{cor:banded}
Let each block $\bA_k$ have upper bandwidth $u_k$ and lower bandwidth $\ell_k$, and set
$b_u = \max_k u_k$, $b_\ell = \max_k\ell_k$.  Then,
\begin{equation*}
    r_2 \le n + \sum_{i=1}^{b_u}(n-i) + \sum_{j=1}^{b_\ell}(n-j).
\end{equation*}
\end{corollary}

We can also establish a relationship between the sparsity within the blocks of $\bA$ and the block-sparsity of the permuted matrix $\bB$. Here, block-sparsity refers to the ratio of nonzero blocks to the total number of blocks. To do so, we present the following theorem.

\begin{theorem}
\label{thm:block_sparse}
   Let $\bA$ be a block-structured matrix and let $\bB$ be the matrix permuted from $\bA$ such that the inner and outer structures have been swapped. Define the sets $C_k$ as above and similarly define the sets $D_\kappa$ corresponding to the blocks of $\bB$, i.e., $D_\kappa := \{(i,j)\mid[\bB_\kappa]_{ij}\ne 0\}$.   Finally, let $\eta_k$ and $\xi_\kappa$ be the number of times that block $\bA_k$ and $\bB_\kappa$ repeat in $\bA$ and $\bB$, respectively.  Then, 
\begin{equation*}   \mathtt{card}\!\Bigl(\bigcup_{k=1}^p C_k\Bigr) = \sum_{\kappa=1}^\rho\xi_\kappa,
    \qquad
    \mathtt{card}\!\Bigl(\bigcup_{\kappa=1}^\rho D_\kappa\Bigr) = \sum_{k=1}^p\eta_k.
\end{equation*}
\end{theorem}

\begin{proof}
We prove the first equality; the second follows by applying the same argument to $\bB$.  Define
$$\overline{\bigcup_{k=1}^{p} C_k} = \{(i,j)\in [m]\times [n]\mid[\bA_k]_{ij}=0\;\mbox{for all } k\}.$$  For each $\kappa$, let
$\{(\alpha,\beta)\}_\kappa = \{(\alpha,\beta)\mid\bB^{(\alpha,\beta)}=\bB_\kappa\}$, so
$\mathtt{card}(\{(\alpha,\beta)\}_\kappa) = \xi_\kappa$.

\emph{Step 1.}  If $(\alpha,\beta)\in\overline{\bigcup_k C_k}$, then
$[\bB^{(\alpha,\beta)}]_{\gamma\delta} = [\bA^{(\gamma,\delta)}]_{\alpha\beta} = 0$ for all
$(\gamma,\delta)$, so $\bB^{(\alpha,\beta)}=\bzero\notin B$.

\emph{Step 2.}  We claim the sets $\{(\alpha,\beta)\}_\kappa$ partition $\bigcup_k C_k$.
\emph{Coverage:} if $(\alpha,\beta)\in\bigcup_k C_k$, then some entry of $\bB^{(\alpha,\beta)}$
is nonzero (by Lemma~\ref{lemma:permutation}), so $\bB^{(\alpha,\beta)}\in B$ and
$(\alpha,\beta)\in\{(\alpha,\beta)\}_\kappa$ for some $\kappa$.  For the reverse, let $(\alpha,\beta)\in\{(\alpha,\beta)\}_{\kappa}$ for an arbitrary choice of $\kappa$. Suppose that $(\alpha,\beta)\in\overline{\bigcup_k C_{k}}$. Since $(\alpha,\beta)\in\{(\alpha,\beta)\}_{\kappa}$, there exists some $\kappa\in[\rho]$ such that $\bB^{(\alpha,\beta)} = \bB_{\kappa}$. However, this contradicts Step~1. Thus, $(\alpha,\beta)$ must not be in the set $\overline{\bigcup_k C_{k}}$ and then must be in the set $\bigcup_k C_{k}$. \emph{Disjointness:} if $(\alpha,\beta)\in\{(\alpha,\beta)\}_{\kappa_1}\cap
\{(\alpha,\beta)\}_{\kappa_2}$ then $\bB_{\kappa_1}=\bB^{(\alpha,\beta)}=\bB_{\kappa_2}$,
contradicting distinctness for $\kappa_1\ne\kappa_2$.

Hence $\mathtt{card}(\bigcup_k C_k) = \sum_\kappa\mathtt{card}(\{(\alpha,\beta)\}_\kappa)
= \sum_\kappa\xi_\kappa$.
\end{proof}

\section{Numerical Examples}
\label{sec:num_results}

\subsection{BTHTHB Matrices: Verifying the Theory}
\label{sec:t_plus_h}

Matrices with \linebreak Block-Toeplitz-plus-Hankel structure with Toeplitz-plus-Hankel blocks (BTHTHB) arise in image
deblurring~\cite{hansen2006deblur,kilmer2007tplush,nagy2004reflexivekron}.  We verify
Corollary~\ref{cor:nested} by comparing the theoretical bound $r_{th}$ (derived from structure)
with the numerically computed rank $r_{calc}=\mathrm{rank}(\munfold{A}{2})$ for four BTHTHB
structures.

Let $\bA\in\RR{n^2}{n^2}$ be BTHTHB with $n\times n$ blocks and \emph{random entries}, so that
any dependence among the blocks of $\bA$ is attributable to structure alone.  Write
$\bA = \bA_T+\bA_H$ where $\bA_T$ is block-Toeplitz and $\bA_H$ is block-Hankel.  Letting
$\nu\sim\mathcal{N}(0,1)$ independently, we construct
\begin{equation*}
    \bA = \Bigl(\sum_i\tilde{\bE}_{T_i}\otimes\sum_j\nu_j\tilde{\bE}_{T_j}\Bigr)
         +\Bigl(\sum_k\tilde{\bE}_{H_k}\otimes\sum_\ell\nu_\ell\tilde{\bE}_{H_\ell}\Bigr),
\end{equation*}
where $\tilde{\bE}_{(\cdot)}$ are the unscaled location-tally matrices.  The four structures
tested (illustrated in Figure~\ref{fig:t_plus_h_fig}) are:
\begin{itemize}
    \item \emph{Symmetric Full:} $\bA_T$ full and block-symmetric and $\bA_H$ full and
          \linebreak block-persymmetric.
    \item \emph{Nonsymmetric Full:} same, without symmetry or persymmetry.
    \item \emph{Symmetric Banded:} $\bA_T$ block-banded with upper bandwidth $n/2-1$, lower
          bandwidth $n/2$ (nearly symmetric); $\bA_H$ analogously banded (nearly persymmetric).
    \item \emph{Nonsymmetric Banded:} as above, without the symmetry conditions.
\end{itemize}

The results are shown in Table~\ref{tab:t+h_table}. In the center columns, we compare $r_{th}$ to $p$ and $r_{calc}$, the rank of $\munfold{A}{2}$ as calculated by the MATLAB $\mathtt{rank}$ command. Agreement $r_{th}=r_{calc}$ holds in all cases, confirming that our structural bounds are exact. The relative error between $\bA$ and the reconstruction $\tilde{\bA}$ using only $r_{th}$ Kronecker terms was on the order of $10^{-15}$ in all cases.

\begin{figure}[t]
    \centering
    \includegraphics[width=\textwidth]{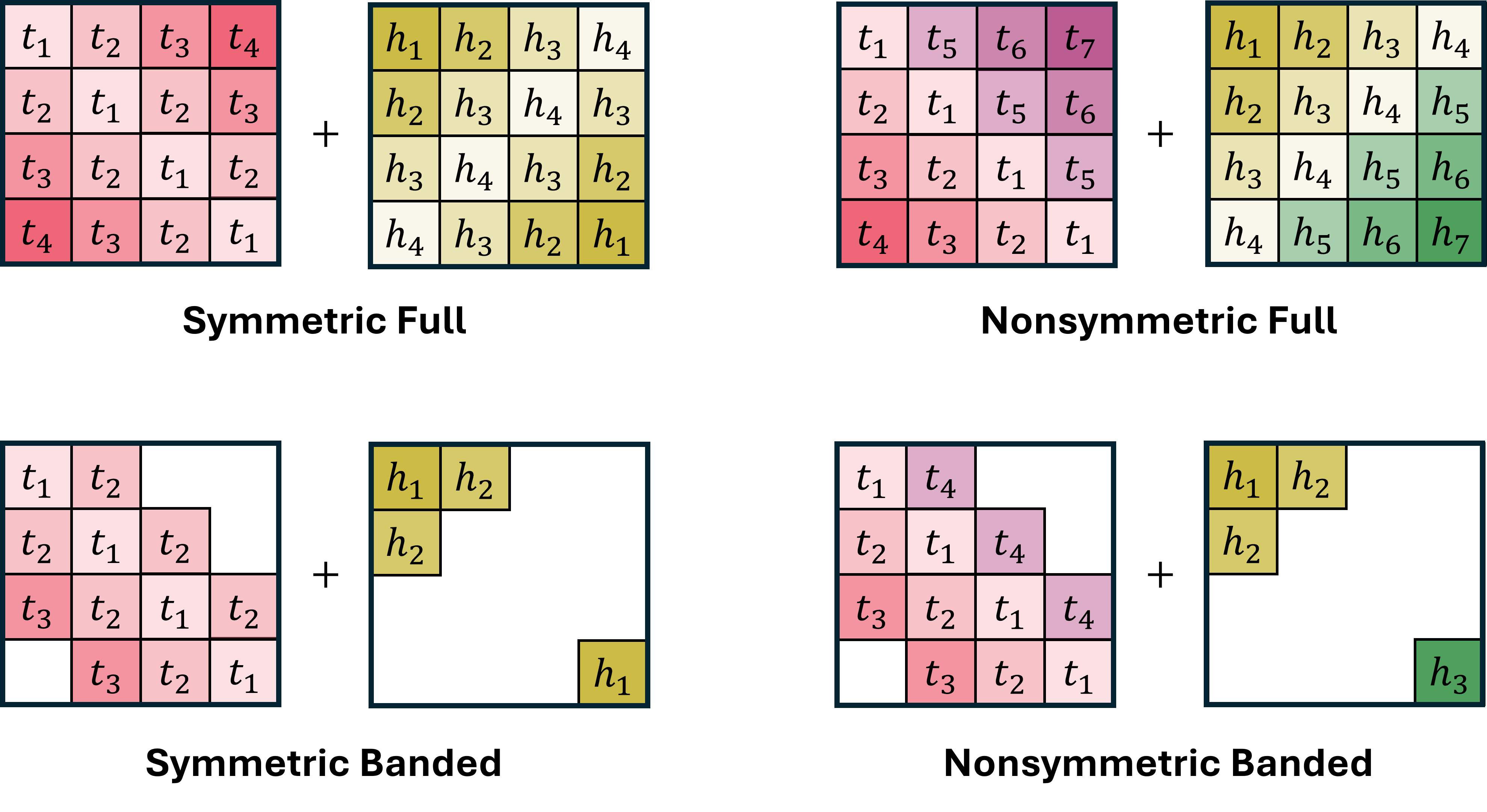}
    \caption{The four BTHTHB structures for $n=4$.}
    \label{fig:t_plus_h_fig}
\end{figure}

\begin{table}[t]
    \centering
    \begin{tabular}{c c c c c c c}
    \cmidrule[1pt]{1-7}
    \textbf{Structure} && $n$ && $p$ & $r_{th}$ & $r_{calc}$\\
    \cmidrule{1-1}\cmidrule{3-3}\cmidrule{5-7}
    \multirow{4}*{\begin{tabular}{c}\textbf{Symmetric Full}\\$p=\frac{n^2}{4}+\frac{n}{2}$\\
                  $r_{th}=2n-2$\end{tabular}}
    &&\textbf{8}&&20&14&14\\&&\textbf{16}&&72&30&30\\&&\textbf{32}&&272&62&62\\&&\textbf{64}&&1056&126&126\\
    \cmidrule{1-1}\cmidrule{3-3}\cmidrule{5-7}
    \multirow{4}*{\begin{tabular}{c}\textbf{Nonsymmetric Full}\\$p=n^2$\\
                  $r_{th}=4n-4$\end{tabular}}
    &&\textbf{8}&&64&28&28\\&&\textbf{16}&&256&60&60\\&&\textbf{32}&&1024&124&124\\&&\textbf{64}&&4096&252&252\\
    \cmidrule{1-1}\cmidrule{3-3}\cmidrule{5-7}
    \multirow{4}*{\begin{tabular}{c}\textbf{Symmetric Banded}\\$p=\frac{n^2}{16}+\frac{3n}{4}+1$\\
                  $r_{th}=n+1$\end{tabular}}
    &&\textbf{8}&&11&9&9\\&&\textbf{16}&&29&17&17\\&&\textbf{32}&&89&33&33\\&&\textbf{64}&&305&65&65\\
    \cmidrule{1-1}\cmidrule{3-3}\cmidrule{5-7}
    \multirow{4}*{\begin{tabular}{c}\textbf{Nonsymmetric Banded}\\$p=\frac{n^2}{4}+n$\\
                  $r_{th}=2n-1$\end{tabular}}
    &&\textbf{8}&&24&15&15\\&&\textbf{16}&&80&31&31\\&&\textbf{32}&&288&63&63\\&&\textbf{64}&&1088&127&127\\
    \bottomrule
    \end{tabular}
    \caption{Comparison of structural bound $r_{th}$ and numerical rank $r_{calc}$ for four
    BTHTHB structures.  Each $\bA\in\RR{n^2}{n^2}$ has $\tens{A}\in\RRR{n}{p}{n}$ and
    $\munfold{A}{2}\in\RR{p}{n^2}$.  The equality $r_{th}=r_{calc}$ holds in all cases.}
    \label{tab:t+h_table}
\end{table}

\subsection{Sparsity Bounds}
\label{sec:sparse_ex}

We now illustrate Theorem~\ref{thm:sparse} for small examples with $q=4$ and $n=8$.

\paragraph{Shared sparsity pattern}
Let $\bA\in\RR{n^2}{n^2}$ be a block-dense matrix whose blocks all share the same sparsity
pattern at level $\tau=0.3$, with entries drawn independently from $\mathcal{N}(0,1)$.  Since
the blocks share a single pattern, $\bigcup_k C_k = C_1$ and the sparsity bound gives
$c = \lfloor\tau n^2\rfloor = 19$.  Because there is no repeated outer block structure, the outer
block space is the full space of $q\times q$ matrices, giving $\dim(\mathscr{B}) = q^2 = 16$ as
a crude structural upper bound.  Here, the structural bound is tighter: $r_2 \le \min\{19,16\}=16$.

For $\tau = 0.2$ the sparsity bound becomes $c = \lfloor 0.2\cdot 64\rfloor = 12$, while the
structural bound remains $q^2=16$.  Now sparsity governs: $r_2\le 12$.

\paragraph{Differing sparsity patterns}
Finally, suppose each block has a distinct sparsity pattern, but each has exactly 7 nonzero
entries.  The union $\bigcup_k C_k$ is then precisely the set of positions that are nonzero in
at least one block; if the union has the same cardinality as in the $\tau=0.2$ case above, the
bound $r_2\le 12$ applies equally.

Figure~\ref{fig:sparse} shows spy plots of $\bA$ (left column) and $\munfold{A}{2}$ (right
column) for each case, confirming the predicted number of nonzero columns.

\begin{figure}[t!]
    \centering
    \begin{subfigure}{0.43\linewidth}
        \centering
        \hspace{-2cm}
        \includegraphics[width = 1.2\textwidth]{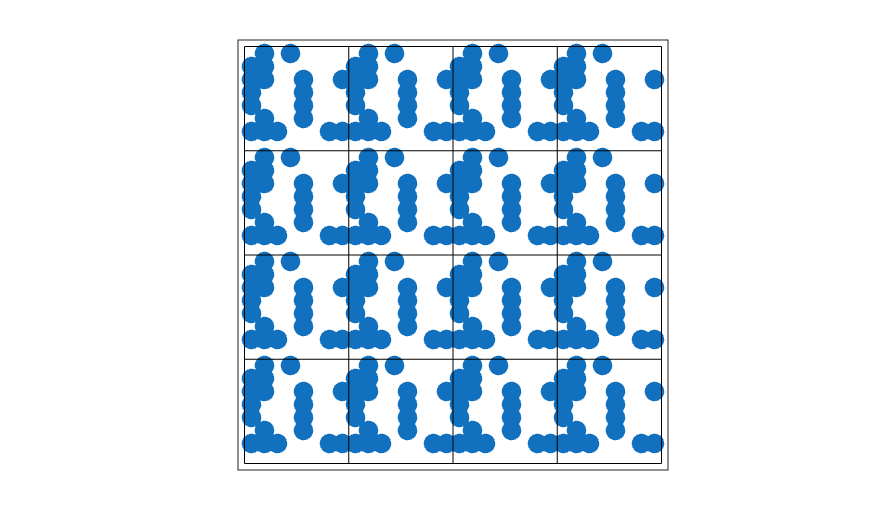}
    \end{subfigure}
    \begin{subfigure}{0.43\linewidth}
        \centering
        \hspace{-2cm}
        \includegraphics[width = 1.2\textwidth]{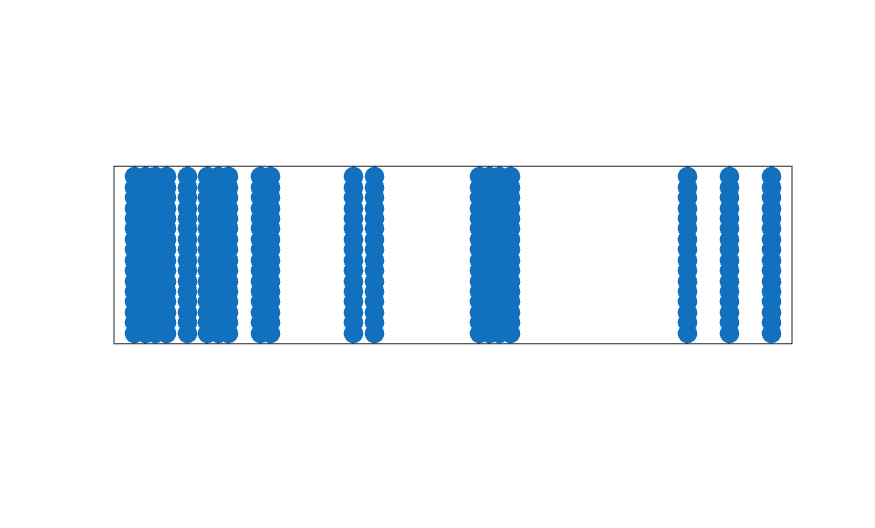}
    \end{subfigure}
    \begin{subfigure}{0.43\linewidth}
        \centering
        \hspace{-2cm}
        \includegraphics[width = 1.2\textwidth]{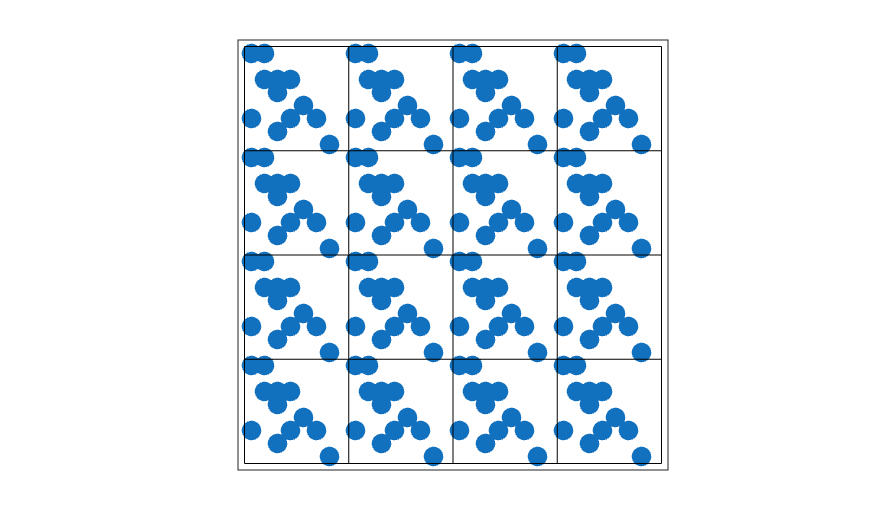}
    \end{subfigure}
    \begin{subfigure}{0.43\linewidth}
        \centering
        \hspace{-2cm}
        \includegraphics[width = 1.2\textwidth]{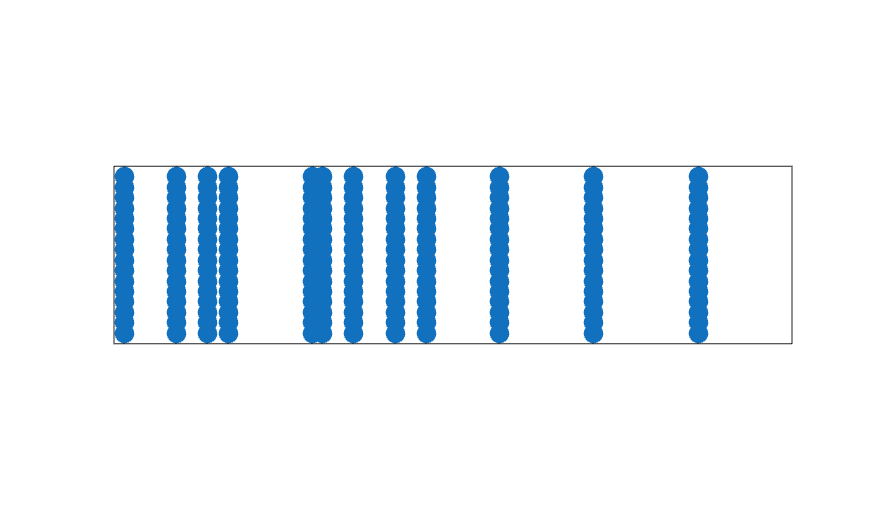}
    \end{subfigure}
    \begin{subfigure}{0.43\linewidth}
        \centering
        \hspace{-2cm}
        \includegraphics[width = 1.2\textwidth]{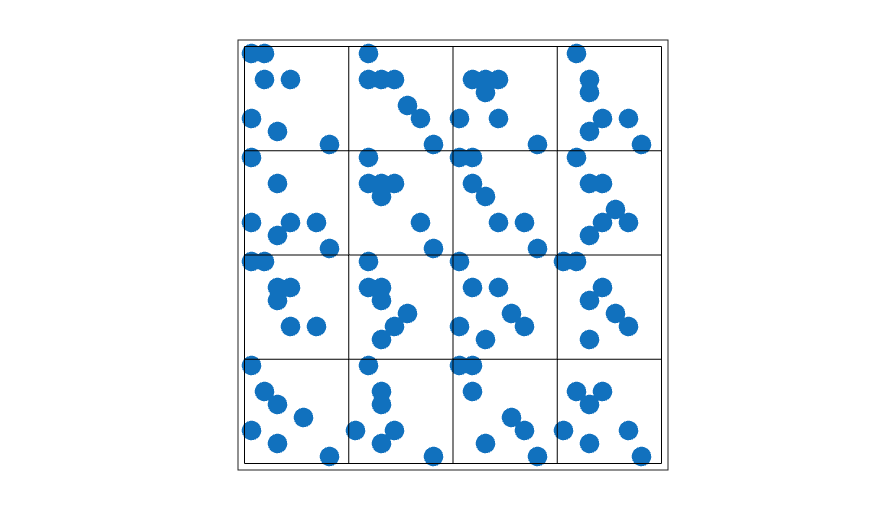}
    \end{subfigure}
    \begin{subfigure}{0.43\linewidth}
        \centering
        \hspace{-2cm}
        \includegraphics[width = 1.2\textwidth]{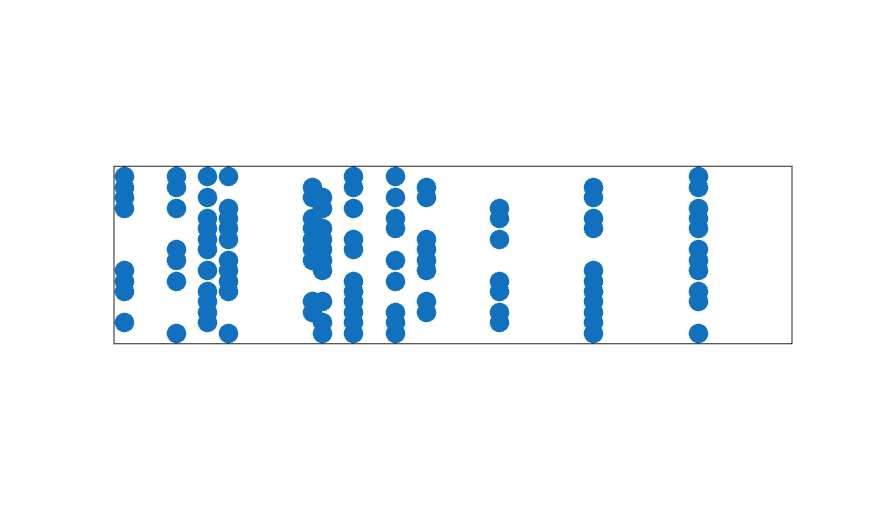}
    \end{subfigure}
    \caption{Left column: spy plots of $\bA$. Right column: corresponding spy plots of
    $\munfold{A}{2}$, confirming the predicted number of nonzero columns in each case.
    Top row ($\tau=0.3$, shared pattern): the structural bound $q^2=16$ is tighter than the
    sparsity bound $c=19$.  Middle row ($\tau=0.2$, shared pattern): the sparsity bound
    $c=12$ governs.  Bottom row (differing patterns, union of cardinality 12): the sparsity
    bound again gives $r_2\le 12$.}
    \label{fig:sparse}
\end{figure}

\subsection{Explaining Singular Value Decay in SuiteSparse Matrices}
\label{sec:suitsparse}

We now revisit two examples from~\cite{kilmer2022matrixtensordecomp} in which the authors demonstrate compression of sparse matrices from the SuiteSparse collection~\cite{davis2011suite}. While that work attributes the compression to rapid decay of the singular values of $\munfold{A}{2}$, it does not explain \emph{why} the decay occurs. Our theory provides that explanation.

\paragraph{Matrix \texttt{t2d\_q4}}
This finite-difference matrix for nonlinear diffusion is sparse and block-tridiagonal. It is block $99\times 99$, with each block also of size $99\times 99$. The associated tensor satisfies $\tens{A}\in\RRR{99}{295}{99}$; a naive mapping would yield 295 Kronecker terms. Kilmer and Saibaba showed that compressing the second mode to rank 5 caused negligible error.

The matrix $\bA$ has the form
\begin{equation*}
    \bA = \left[\begin{matrix}
        \bA_{d_1} & \bA_{o_3} & &\\
        \bA_{o_2} & \bA_{d_4} & \bA_{o_6} &\\
        & \ddots & \ddots &  \ddots &\\
        & & \bA_{o_{290}} & \bA_{d_{292}} & \bA_{o_{294}} \\
        & & & \bA_{o_{293}} & \bA_y
    \end{matrix}\right],
\end{equation*}
where
\begin{equation*}
    \bA_{d_k} = \left[\begin{matrix}
        8/3 & -1/3 & \\
        -1/3 & \ddots & \ddots \\
        & \ddots & 8/3 &  -1/3 \\
        & & -1/3 & x_k 
    \end{matrix}\right],
    \qquad
    \bA_{o_k} = \left[\begin{matrix}
        -1/3 & -1/3 & \\
        -1/3 & \ddots & \ddots \\
        & \ddots & -1/3 &  -1/3 \\
        & & -1/3 & x_k 
    \end{matrix}\right],
\end{equation*}
and $\bA_y$ has tridiagonal sparsity with entries $y_1,\dots,y_{295}$ that need not follow the
$8/3$, $-1/3$ pattern.

We now identify an explicit basis for the subspace containing all blocks of $\bA$.  Define the
following three $99\times 99$ matrices:
\begin{itemize}
    \item $\bPhi_d$: the tridiagonal matrix with $8/3$ on the superdiagonal and subdiagonal and
          on the first 98 diagonal entries, and $0$ at position $(99,99)$;
    \item $\bPhi_o$: the tridiagonal matrix with $-1/3$ everywhere on the super- and
          subdiagonals and on the first 98 diagonal entries, and $0$ at position $(99,99)$;
    \item $\bPhi_c$: the matrix that is $0$ everywhere except for the value $1$ at position
          $(99,99)$.
\end{itemize}
Every diagonal block satisfies $\bA_{d_k} = \bPhi_d + x_k\bPhi_c$, and every off-diagonal block
satisfies $\bA_{o_k} = \bPhi_o + x_k\bPhi_c$.  Thus $\{\bPhi_d, \bPhi_o, \bPhi_c\}$ spans all
blocks of $\bA$ except possibly $\bA_y$.  These three matrices are linearly independent: $\bPhi_c$ is supported only
at position $(99,99)$, while $\bPhi_d$ and $\bPhi_o$ are zero there and differ on the interior
diagonal entries ($8/3$ vs.\ $-1/3$), so no nontrivial linear combination can vanish.

If $\bA_y \in \mathrm{span}\{\bPhi_d,\bPhi_o,\bPhi_c\}$, then all blocks lie in a
three-dimensional subspace and $r_2 \le 3$.  If $\bA_y \notin
\mathrm{span}\{\bPhi_d,\bPhi_o,\bPhi_c\}$, it contributes exactly one additional dimension.
In either case,
\begin{equation*}
    \mathscr{A} \subseteq \mathrm{span}\{\bPhi_d,\,\bPhi_o,\,\bPhi_c,\,\bA_y\},
\end{equation*}
giving $\dim(\mathscr{A})\le 4$, and by Theorem~\ref{thm:r2=dimC}, $r_2\le 4$.

Figure~\ref{fig:svds} (left) plots the first 20 singular values of $\munfold{A}{2}$, confirming that the numerical rank is exactly 4.  Our structural analysis thus provides a tight bound and fully explains the rapid singular value decay reported in~\cite{kilmer2022matrixtensordecomp}.

\paragraph{Matrix \texttt{fv2}}
We now consider the matrix $\mathtt{fv2}$ from the SuiteSparse collection, which has the form

\begin{equation*}
    \bA = \left[\begin{matrix}
        \bA_d & \bA_o & &\\
        \bA_o & \bA_d & \bA_o &\\
        & \ddots & \ddots &  \ddots &\\
        & & \bA_o & \bA_d & \bA_o \\
        & & & \bA_o & \bA_d
    \end{matrix}\right],
\end{equation*}
where the only two distinct blocks $\bA_d$ and $\bA_o$ are
\begin{small}
\begin{equation*}
    \bA_d = \left[\begin{matrix}
        3.5101 & -0.5 & \\
        -0.5 & \ddots & \ddots \\
        & \ddots & \ddots & -0.5  \\
        & & -0.5 & 3.5101 
    \end{matrix}\right]
    \quad\mbox{and} \quad
    \bA_o = \left[\begin{matrix}
        0.5 & -0.25 & \\
        -0.25 & \ddots & \ddots \\
        & \ddots & \ddots & -0.25  \\
        & & -0.25 & 0.5 
    \end{matrix}\right].
\end{equation*}
\end{small}
This matrix is block-symmetric and block-tridiagonal, with symmetric tridiagonal blocks whose diagonals are constant. The outer structure is therefore a symmetric Toeplitz tridiagonal pattern of dimension 2, and the inner structure has the same dimension. Hence $r_2 = \dim(\mathscr{A}) = 2$, and $\bA$ can be written exactly as a sum of two Kronecker products. Figure~\ref{fig:svds} (right) confirms that $\munfold{A}{2}$ has exactly two nonzero singular values.

\begin{figure}[t]
        \centering
        \begin{subfigure}{0.45\linewidth}
        \centering
            \includegraphics[width = 0.95\textwidth]{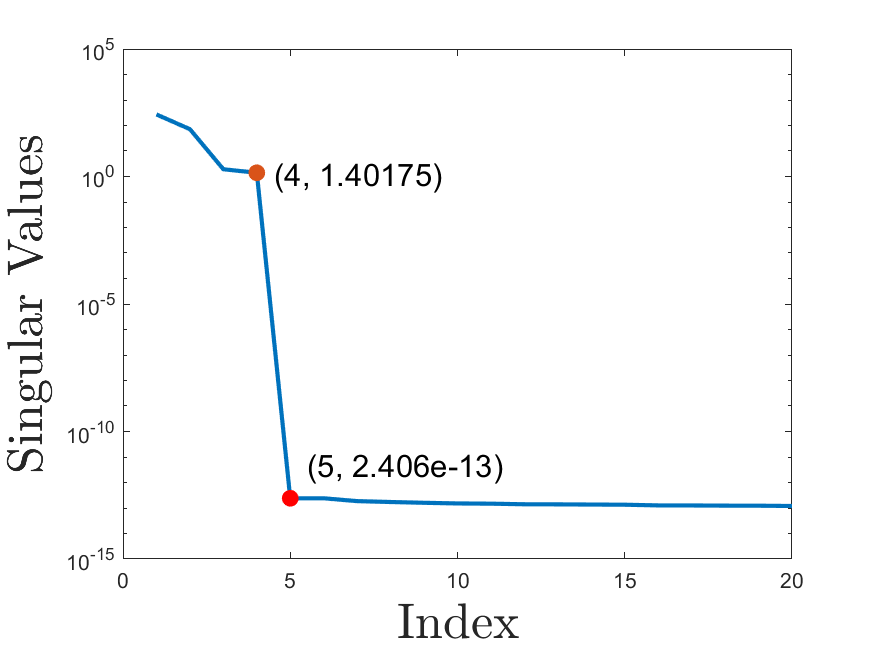}
        \label{fig:sv_t2d_q4}
        \end{subfigure}
        \begin{subfigure}{0.45\linewidth}
        \centering
            \includegraphics[width = 0.95\textwidth]{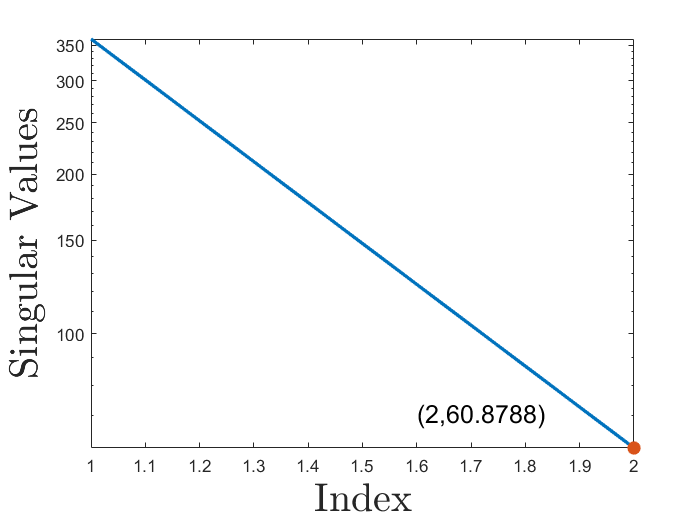}
        \label{fig:sv_fv2}
        \end{subfigure}
        \caption{Left: First 20 singular values of $\munfold{A}{2}$ for \texttt{t2d\_q4}. Right: Singular values of $\munfold{A}{2}$ for \texttt{fv2}.}
        \label{fig:svds}
    \end{figure}

\section{Conclusion}
\label{sec:conclusion}

We have derived a complete characterization of the Kronecker rank of a block-structured matrix in
terms of its underlying structure.  The central result is the chain
$r_2 = \dim(\mathscr{A}) = \dim(\mathscr{B})$, proved both by a direct dimension argument and via
an explicit isomorphism between the outer blockspan and the column space of the second-mode
unfolding.  The nested containments $\mathscr{A}\subseteq\mathscr{S}\subseteq\mathscr{F}$ and
$\mathscr{B}\subseteq\mathscr{T}\subseteq\mathscr{E}$ translate structural and sparsity
information into computable upper bounds on $r_2$.  For two sparse matrices from the SuiteSparse
collection, these bounds are tight and provide a structural explanation for singular value decay
that had previously been observed but not understood.

Several directions for future work remain open.  First, extending the isomorphism proof to
matrices with \emph{multiple levels} of nested block structure would generalize the framework to
higher-order Tucker decompositions.  Second, the preliminary analysis of modes~1 and~3 of the
associated tensor, which can exhibit rank deficiency when the blocks of $\bA$ are themselves
low-rank, warrants further investigation; a full characterization of this case could enable
additional compression beyond the second mode.  Third, a systematic study of the conditions under
which $\dim(\mathscr{A})<\dim(\mathscr{S})$ (i.e., when the actual Kronecker rank is strictly
below the structural bound) would be practically useful.

\section*{Acknowledgments}
ME was supported through a Karen EDGE Fellowship. MK acknowledges support from NSF DMS 2410698.

\appendix
\section{A Worked Example: Toeplitz-plus-Hankel Outer Structure}
\label{app:t+h_worked}

We revisit the remark from Section~\ref{subsec:structural_bounds}. We demonstrate why $\mathscr{T}\subsetneq\mathscr{E}$ for a matrix with a block $5\times 5$
symmetric-Toeplitz-plus-persymmetric-Hankel outer structure and show in this case
 that $\dim(\mathscr{E})=9$ while $\dim(\mathscr{T})=8$.

\paragraph{Counting distinct blocks and deriving $\dim(\mathscr{E})$}
A $5\times 5$ symmetric Toeplitz matrix $\bA_T$ satisfies $[\bA_T]_{\gamma\delta} =
t_{|\gamma-\delta|}$ for parameters $t_0,\dots,t_4$, giving 5 free parameters.  A $5\times 5$
persymmetric Hankel matrix $\bA_H$ satisfies $[\bA_H]_{\gamma\delta} = h_{\gamma+\delta-1}$ for
parameters $h_1,\dots,h_5$, also giving 5 free parameters.  In the sum $\bC = \bA_T + \bA_H$,
an entry $[\bC]_{\gamma\delta}$ depends on the pair $(|\gamma-\delta|, \gamma+\delta-1)$.
Table~\ref{tab:pairs} lists all distinct such pairs for $\gamma,\delta\in\{1,\dots,5\}$; direct
enumeration yields exactly 9 distinct pairs (the pair $(0,5)$ at $\gamma=\delta=3$ is the one
whose constraint is derived below), giving $p=9$ distinct block positions.  Since the
location-tally matrices $\{\bE_k\}_{k=1}^9$ each record a distinct pattern of block positions
and are therefore linearly independent, $\dim(\mathscr{E}) = 9$.  For general even $q$, the
analogous count gives $\dim(\mathscr{E}) = q^2/4 + q/2$.

\begin{table}[h]
\centering
\small
\begin{tabular}{ccc}
\toprule
Position $(\gamma,\delta)$ (representative) & $|\gamma-\delta|$ & $\gamma+\delta-1$ \\
\midrule
$(1,1)$ & 0 & 1 \\
$(1,2),(2,1)$ & 1 & 2 \\
$(1,3),(3,1)$ & 2 & 3 \\
$(1,4),(4,1)$ & 3 & 4 \\
$(1,5),(5,1)$ & 4 & 5 \\
$(2,2)$ & 0 & 3 \\
$(2,3),(3,2)$ & 1 & 4 \\
$(2,4),(4,2)$ & 2 & 5 \\
$(3,3)$ & $0$ & $5$ \\
\bottomrule
\end{tabular}
\caption{The 9 distinct $(|\gamma-\delta|,\,\gamma+\delta-1)$ pairs for a $5\times 5$
symmetric-Toeplitz-plus-persymmetric-Hankel matrix, confirming $\dim(\mathscr{E})=9$.
The remaining positions, such as $(2,5),(3,4),(3,5),(4,5)$ and their transposes, produce
pairs such as $(3,6),(1,6),(2,7),(3,8)$ that are already listed above: Toeplitz symmetry
identifies $t_{|\gamma-\delta|}$ with the same parameter for transposed positions, and Hankel
persymmetry identifies $h_{\gamma+\delta-1}$ for positions with the same antidiagonal sum.
These structural identifications reduce the 25 block positions to the 9 listed pairs without
introducing new independent parameters.}
\label{tab:pairs}
\end{table}

\paragraph{The linear constraint and $\dim(\mathscr{T})$}
Label the nine distinct entry values of a generic $\bC = \bA_T + \bA_H$ as
$\gamma_i = t_{i-1} + h_{6-i}$ for $i=1,\dots,5$, and
$\gamma_6 = t_1+h_3$, $\gamma_7 = t_2+h_2$, $\gamma_8 = t_3+h_1$, $\gamma_9 = t_0+h_4$.
(Here $\gamma_9$ is the center entry $[\bC]_{3,3} = t_0+h_4$.)
These nine values are not independent: they satisfy the linear relation
\begin{equation}
\label{eq:t+h_constraint}
\gamma_6 + \gamma_8 - \gamma_3 = (t_1+h_3)+(t_3+h_1)-(t_3+h_3) = t_1+h_1 = \gamma_9,
\end{equation}
so $\dim(\mathscr{T}) = 8$.  In contrast, $\mathscr{E}$ imposes no such relation, giving
$\dim(\mathscr{E}) = 9 > \dim(\mathscr{T}) = 8$.  The matrix $\bX$ obtained from $\bC$ by
replacing the third-column (and third-row) entries with an independent value $t_6+h_6 \ne
t_3+h_3$ lies in $\mathscr{E}$ but violates~\eqref{eq:t+h_constraint}, confirming
$\mathscr{T}\subsetneq\mathscr{E}$.

\begin{lemma}
\label{lem:t+h_constraint}
Let $\bX\in\mathscr{E}$ with entries labeled $\gamma_1,\dots,\gamma_9$ as above.  Then
$\bX\in\mathscr{T}$ if and only if $\gamma_9 = \gamma_6 + \gamma_8 - \gamma_3$.
\end{lemma}

\begin{proof}
\emph{Only if.}  If $\bX = \bT+\bH$ with $\bT$ symmetric Toeplitz and $\bH$ persymmetric
Hankel, then the labeling above gives~\eqref{eq:t+h_constraint} directly.

\emph{If.}  Suppose $\gamma_9 = \gamma_6 + \gamma_8 - \gamma_3$.  We construct explicit
Toeplitz and Hankel parameters.  Set
\begin{equation}
\label{eq:toeplitz_params}
    t_0 = \gamma_1,\quad t_1 = \gamma_6 - h_3,\quad t_2 = \gamma_7 - h_2,\quad
    t_3 = \gamma_8 - h_1,\quad t_4 = \gamma_5 - h_1,
\end{equation}
where the Hankel parameters are determined by
\begin{equation}
\label{eq:hankel_params}
    h_5 = \gamma_1 - t_0,\quad h_4 = \gamma_2 - t_1,\quad h_3 = \gamma_3 - t_2,\quad
    h_2 = \gamma_4 - t_3,\quad h_1 = \gamma_5 - t_4.
\end{equation}
We must verify that this system is consistent.  From~\eqref{eq:hankel_params}, $h_3 = \gamma_3
- t_2$ and $h_1 = \gamma_5 - t_4$.  Substituting into $t_1 = \gamma_6 - h_3$ gives
$t_1 = \gamma_6 - \gamma_3 + t_2$, and substituting into $t_3 = \gamma_8 - h_1$ gives
$t_3 = \gamma_8 - \gamma_5 + t_4$.  These are two equations in $t_0,\dots,t_4$ and
$h_1,\dots,h_5$.  The system has a free parameter (say $t_0$); choosing $t_0 = \gamma_1 - h_5$
with $h_5 = \gamma_1 - t_0$ means $h_5$ is determined once $t_0$ is chosen.  Setting
$h_5 = 0$ fixes $t_0 = \gamma_1$, from which all other parameters follow uniquely.  Under the
constraint $\gamma_9 = \gamma_6 + \gamma_8 - \gamma_3$, one verifies directly that
$[\bT]_{33} = t_0$ and $[\bH]_{33} = h_4 = \gamma_9 - t_0$, giving $[\bT+\bH]_{33} = \gamma_9$,
as required.  All other entries follow from the Toeplitz and Hankel structure by
construction.  Therefore $\bX = \bT + \bH\in\mathscr{T}$.
\end{proof}

\bibliographystyle{plain}
\bibliography{references}

\end{document}

%% file: settings.tex
\usepackage{graphicx}
\usepackage{amsfonts}
\usepackage[mathscr]{eucal}
\usepackage{booktabs}
\usepackage{multirow}
\usepackage{amssymb}
\usepackage{subcaption}
\usepackage[normalem]{ulem}

\newsiamremark{remark}{Remark}
\newsiamremark{example}{Example}

\usepackage{subcaption}

\newcommand{\tens}[1]{\mathbf{\mathcal{#1}}}

\newcommand{\RR}[2]{\mathbb{R}^{#1 \times #2}}
\newcommand{\RRR}[3]{\mathbb{R}^{#1 \times #2 \times #3}}

\newcommand{\munfold}[2]{\mathbf{#1}_{(#2)}}
\newcommand{\lat}[2]{\overrightarrow{\tens{#1}}_#2}
\newcommand{\frontal}[2]{\mathbf{#1}^{[#2]}}

\def\bA{\mathbf{A}}
\def\bB{\mathbf{B}}
\def\bC{\mathbf{C}}
\def\bD{\mathbf{D}}
\def\bE{\mathbf{E}}
\def\bF{\mathbf{F}}
\def\bG{\mathbf{G}}
\def\bH{\mathbf{H}}
\def\bI{\mathbf{I}}

\def\bP{\mathbf{P}}
\def\bQ{\mathbf{Q}}

\def\bS{\mathbf{S}}
\def\bT{\mathbf{T}}
\def\bU{\mathbf{U}}
\def\bV{\mathbf{V}}
\def\bW{\mathbf{W}}
\def\bX{\mathbf{X}}
\def\bY{\mathbf{Y}}
\def\bZ{\mathbf{Z}}

\def\ba{\mathbf{a}}

\def\bc{\mathbf{c}}

\def\be{\mathbf{e}}
\def\bf{\mathbf{f}}

\def\bv{\mathbf{v}}

\def\bx{\mathbf{x}}

\def\bSigma{\mathbf{\Sigma}}

\def\bPhi{\mathbf{\Phi}}

\def\bzero{\mathbf{0}}